\documentclass[12pt]{amsart}

\usepackage{amsmath,amsthm,amsfonts,amssymb,csquotes,graphicx, float, caption,wrapfig, mathrsfs, mathtools, enumerate}
\usepackage[hidelinks]{hyperref}
\usepackage{tikz}
\usepackage{tikz-cd}
\usetikzlibrary{cd,calc}
\usepackage{xcolor}
\usepackage{accents}
\usepackage{marvosym}
\usepackage[alphabetic,abbrev,lite]{amsrefs}
\usepackage[american]{babel}
\usepackage[left=1in,top=1in,right=1in]{geometry}
\usepackage[capitalize]{cleveref}
\usepackage{subcaption}
\usepackage[all]{xy}
\usepackage{enumitem}
\usepackage{url,stmaryrd,wasysym}

\newtheorem{mainthm}{Theorem}
\Crefname{mainthm}{Theorem}{Theorems}

\usepackage{chngcntr}
\counterwithin{figure}{section}
\numberwithin{equation}{section}

\newtheorem{lemma}{Lemma}[section]
\newtheorem{theorem}[lemma]{Theorem}
\newtheorem{prop}[lemma]{Proposition}
\newtheorem{cor}[lemma]{Corollary}

\newtheorem*{thmA}{Theorem A}
\newtheorem*{thmB}{Theorem B}
\newtheorem*{thmC}{Theorem C}

\theoremstyle{definition}
\newtheorem{defn}[lemma]{Definition}
\newtheorem{notation}[lemma]{Notation}

\newtheorem{examplex}[lemma]{Example}
\newenvironment{example}
  {\pushQED{\qed}\examplex}
  {\popQED\endexamplex}

\theoremstyle{remark}
\newtheorem{remark}[lemma]{Remark}

\newtheorem*{claim*}{Claim} \newtheorem*{case*}{Case}

\newcommand{\Manoa}{M\=anoa}
\newcommand{\Hawaii}{Hawai\kern.05em`\kern.05em\relax i}

\renewcommand{\geq}{\geqslant}

\newcommand{\<}{\langle}
\renewcommand{\>}{\rangle}
\newcommand{\beq}{\begin{displaymath}}
\newcommand{\eeq}{\end{displaymath}}

\newcommand{\coloneq}{\ensuremath{\mathrel{\mathop :}=}}
\newcommand{\defi}[1]{\textsf{#1}} 

\newcommand{\checkS}{\check{S}}
\newcommand{\checkB}{\check{B}}
\newcommand{\checkG}{\check{G}}
\newcommand{\sXCar}{\sX_{{\mathbb{Q}}\text{-}\mathrm{Car}}}
\newcommand{\sXcan}{\sX_{\operatorname{can}}}
\newcommand{\sXGC}{\sX_{\Gamma, \QCart}}
\newcommand{\sXGamma}{\sX_\Gamma}
\renewcommand{\AA}{\ensuremath{\mathbb{A}}}
\newcommand{\bA}{\mathbb A}
\newcommand{\CC}{\mathbb{C}}

\newcommand{\NN}{\ensuremath{\mathbb{N}}}
\newcommand{\PP}{\mathbb P}
\newcommand{\QQ}{\mathbb Q}
\newcommand{\R}{\mathbb{R}}
\newcommand{\RR}{\ensuremath{\mathbb{R}}}
\newcommand{\ZZ}{\mathbb Z}

\newcommand{\bba}{\ensuremath{\mathbf{a}}}

\newcommand{\bbp}{\ensuremath{\mathbf{p}}}

\newcommand{\bbv}{\ensuremath{\mathbf{v}}}

\newcommand{\boldzero}{\ensuremath{\mathbf{0}}}

\newcommand{\cF}{\ensuremath{\mathcal{F}}}

\newcommand{\cH}{\mathcal H}

\newcommand{\cO}{{\mathcal O}}

\newcommand{\sX}{{\mathcal X}}
\newcommand{\cX}{\mathcal{X}}

\newcommand{\fm}{\mathfrak{m}}
\newcommand{\fp}{I_p}

\newcommand{\tf}{\widetilde{f}}

\newcommand{\tL}{\widetilde{L}}

\newcommand{\tN}{\widetilde{N}}

\newcommand{\tp}{\widetilde{p}}
\newcommand{\tfp}{\widetilde{I_p}}

\newcommand\tsX{\widetilde{\sX}}

\newcommand{\tsigma}{\widetilde{\sigma}}
\newcommand{\tSigma}{\widetilde{\Sigma}}

\newcommand{\wG}{\widehat{G}}
\newcommand{\wcheckG}{\widehat{\checkG}}
\newcommand{\wH}{\widehat{H}}

\DeclareMathOperator{\QCart}{\mathbb{Q}\text{-}\mathrm{Car}}
\DeclareMathOperator{\coker}{coker}

\newcommand{\Amp}{\operatorname{Amp}}

\newcommand{\can}{\mathrm{can}}
\newcommand{\CDiv}{\operatorname{CDiv}}
\newcommand{\ce}{\coloneqq}

\newcommand{\Cl}{\on{Cl}}

\newcommand\ev{\mathrm{ev}}

\newcommand{\Hom}{\operatorname{Hom}}

\newcommand{\id}{\operatorname{id}}

\newcommand\Mod{\operatorname{Mod}}

\newcommand\on{\operatorname}

\newcommand{\Pic}{\operatorname{Pic}}

\newcommand\QCoh{\operatorname{QCoh}}

\newcommand{\SF}{\mathrm{SF}}
\newcommand{\Span}{\operatorname{span}} 
\newcommand{\Spec}{\operatorname{Spec}}

\newcommand{\xra}{\xrightarrow}

\newcommand{\newterm}{\textsf}

\title[Nonsimplicial toric Nullstellensatz and stacky GKZ theory]{Nonsimplicial toric Nullstellensatz \\ and stacky GKZ theory}

\author[C. Berkesch]{Christine Berkesch}
\address{School of Mathematics, University of Minnesota}
\email{cberkesc@umn.edu}

\author[D. Erman]{Daniel Erman}
\address{Department of Mathematics, University of \Hawaii \, at \Manoa}
\email{erman@hawaii.edu}

\author[D. Favero]{David Favero}
\address{School of Mathematics, University of Minnesota}
\email{favero@umn.edu}

\makeatletter
\renewcommand{\@setkeywords}{%
  \hspace*{-\parindent}{\itshape \keywordsname.}\enspace \@keywords\@addpunct.}
\renewcommand{\@setsubjclass}{%
  \hspace*{-\parindent}{\itshape \subjclassname}\enspace \@subjclass\@addpunct.}
\makeatother

\makeatletter
\renewcommand{\@setkeywords}{%
  \hspace*{-\parindent}{\itshape \keywordsname.}\enspace \@keywords\@addpunct.}
\makeatother
\begin{document}

\begin{abstract}
We introduce a variant of the Cox ring using $\QQ$-Cartier divisors and use this to remedy various deficiencies of nonsimplicial toric varieties. 
Our main applications are: 
Cox's ideal-variety correspondence, an explicit classification of subschemes and sheaves in terms of multigraded modules, an associated toric Deligne-Mumford stack, and a stacky extension of the GKZ/Mori theory of toric varieties.
\end{abstract}


\vspace*{-10mm}
\maketitle

\refstepcounter{section}

\vspace*{-3mm}
\noindent
Projective geometry and graded algebra are classically intertwined.  
Due to pioneering work of Cox \cite{cox:homog} and others~\cite{Mussontori,audin}, there is an enhancement of this algebro-geometric dictionary connecting the geometry of simplicial toric varieties to multigraded algebra.  For example, Cox's toric ideal-variety correspondence \cite{cox:homog}*{Proposition 2.4} directly generalizes Hilbert's Nullstellensatz for $\PP^n$ to simplicial toric varieties.   

\begin{theorem}[Cox's Ideal-Variety Correspondence]
\label{thm:cox}
If $X$ is a \textbf{simplicial} toric variety with Cox ring $S\ce \bigoplus_{d \in \Cl(X)} H^0(X,\cO_X(d))$ and irrelevant ideal $B$, then there is a bijective correspondence between closed subvarieties of $X$ and radical homogeneous ideals  $I \subseteq B \subseteq S$.
\end{theorem}

For nonsimplicial toric varieties, \Cref{thm:cox} simply fails. 
\begin{example}
\label{ex:motivatingEx}
Consider the nonsimplicial fan $\Sigma$ given by the faces of the unique maximal cone on the rays $(0,0,1), (1,0,1), (0,1,1), (1, 1, 1)$. The Cox ring for $X=X_\Sigma$ in this case is the polynomial ring $S = \CC[a,b,c,d]$ with $\deg(a)=\deg(d)=1$ and $\deg(b)=\deg(c)=-1$, and the irrelevant ideal is $B = \< 1\>$. 
\Cref{thm:cox} breaks down in this case.  For instance, the ideals $\<a,d\>$ and $\<b,c\>$ are distinct, both belong to $B$ (this is trivial because $B$ is the unit ideal), and both correspond to the same subvariety of $X$, namely, the torus-fixed point. 
\end{example}

Geometrically, a semiprojective nonsimplicial toric variety\footnote{A semiprojective toric variety is a normal toric variety that has no torus factors and is projective over its affinization; equivalently, it is the toric variety associated to a full dimensional lattice polyhedron, see \cite{CLSToricVarieties}*{Proposition 7.2.9}.} is no longer a geometric quotient, which has consequences beyond the ideal-variety correspondence. 
The main idea of this paper is that results like Cox's ideal-variety correspondence can be remedied in the nonsimplicial case if the Cox ring $S$ and irrelevant ideal $B$ are replaced by a variant that only considers $\QQ$-Cartier degrees. 

\begin{defn}\label{defn:intCoxAndCartierStack}
For a (normal) semiprojective\footnote{Like the Cox ring, the $\QQ$-Cartier Cox ring can be defined in non-toric settings.  However, we only consider semiprojective toric varieties in this article.
} 
toric variety $X$ over $\CC$, 
write $\QCart(X)\subseteq \Cl(X)$ for the subgroup of $\QQ$-Cartier divisors, i.e. the set of divisors $D$ where some integer multiple $mD$ is a Cartier divisor. 
Define the \defi{$\QQ$-Cartier Cox ring} $\checkS$ as the subring of $\QQ$-Cartier-graded pieces of the Cox ring $S$ of $X$: 
\[
\checkS \ce \bigoplus_{d \in \QCart(X)} H^0(X,\cO_X(d)).
\]
The \defi{$\QQ$-Cartier irrelevant ideal} $\checkB \subseteq \checkS$ is $B\cap \checkS$, where $B$ is the irrelevant ideal in $S$.  By \cite{CLSToricVarieties}*{Proposition 15.2.1(b)}, 
$\checkB = \bigoplus_{d \in \Amp(X)\cap{\QCart}(X)} H^0(X,\cO_X(d))$, 
where $\Amp(X)\subseteq \Cl(X)$ are the ample divisors.
\end{defn}

\begin{remark}
    In the literature, there are two distinct objects that are both sometimes called the Cox ring.  The first involves summing over all Weil divisors and the second involves summing over all Cartier divisors. 
The $\QQ$-Cartier Cox ring sits in between these two.
\end{remark}

In the simplicial case, all divisors are $\QQ$-Cartier \cite{CLSToricVarieties}*{Proposition 4.2.7}, so $\checkS = S$ is the usual Cox ring.  Hence, 
our first result is a direct generalization of Cox's simplicial ideal-variety correspondence from \Cref{thm:cox} to the nonsimplicial case. 

\begin{mainthm}[Toric Ideal-Variety Correspondence]
\label{thm:nullstellensatz-var} 
If $X$ is a semiprojective toric variety, not necessarily simplicial, then there is a bijective correspondence between closed subvarieties of $X$ and radical homogeneous ideals  $I \subseteq \checkB \subseteq \checkS$.
\end{mainthm}

\begin{example}
    Returning to Example~\ref{ex:motivatingEx}, $X$ has no nonzero $\QQ$-Cartier divisors.  Thus $\checkS = \CC[a,b,c,d]_0 = \CC[ab,ac,bd,cd]$ and $\checkB=\<1\>\subseteq \checkS$. 
    When intersected with $\checkS$, the two distinct ideals $\< a,d\>$ and $\< b,c\>$ from Example~\ref{ex:motivatingEx} both become $\< ab,ac,bd,cd\>$, which is now the unique radical ideal of $\checkS$ corresponding to the torus-invariant point of $X$.
\end{example}

While Theorem~\ref{thm:nullstellensatz-var} uses $\checkS$ to classify subvarieties in the nonsimplicial case, classifying subschemes or coherent sheaves is more nuanced. 
In the smooth case, subschemes and sheaves are in bijection with ideals and modules on the Cox ring $S$, considered up to irrelevant torsion, but this correspondence fails even in the simplicial case. 

\begin{example}\label{ex:P113-1}
Consider the weighted projective variety  $X=\PP_v(1, 1, 3)$, which arises from the simplicial two-dimensional fan with rays $(1,0)$, $(-1,3)$, and $(0,-1)$. The Cox ring and $\QQ$-Cartier Cox ring coincide, so $\checkS = S = \CC[x,y,z]$, where $\deg(x)=\deg(y)=1$ and $\deg(z) = 3$, with irrelevant ideal $\checkB = B=\<x,y,z\>$.  

Consider the ideals $I_1 = \< x,y \>$, $I_2 = \< x^2, xy, y^2 \>$, and $I_3 = \< x^3, x^2y, xy^2,y^3\>$, which all define the same subscheme of $X$. This can be seen directly by looking at the three open affine charts given by inverting the variables: for all $j$, 
$H^0(D(x),I_j) = H^0(D(y),I_j) = 1$ and 
\[
H^0(D(z),I_j) = \left\< \frac{x^3}{z},\, \frac{x^2y}{z},\, \frac{xy^2}{z},\, \frac{y^3}{z} \right\>.
\qedhere
\]
\end{example}

This problem is resolved by introducing certain toric stacks.\footnote{For readers unaccustomed to stacks, \Cref{sec:background} spells out the connection between multigraded commutative algebra and the global quotient stacks of interest in this article.} 

\begin{example}\label{ex:P113-2}
Returning to \Cref{ex:P113-1}, on the weighted projective stack $\PP(1, 1, 3)\ce [\bA^3-\boldzero/\CC^*]$, 
each $I_j$ defines a distinct substack that corresponds to a single point with different scheme structures.  Specifically, each $I_j$ defines a global quotient space $[\Spec (\CC[x,y,z]/I_j )- 0/\CC^*]$, and there are stack isomorphisms 
\begin{align*}
[\Spec (\CC[x,y,z]/I_1 )- \boldzero/\CC^*] & \cong  [p/\ZZ_3], 
\\
[\Spec (\CC[x,y,z]/I_2 )- \boldzero/\CC^*] & \cong  [2p/\ZZ_3], 
\\
[\Spec (\CC[x,y,z]/I_3 )- \boldzero/\CC^*] & \cong  [3p/\ZZ_3], 
\end{align*}
where $p,2p,3p$ are the usual point and the nonreduced points for thickenings of orders one, two, and three, respectively, defined by these ideals.
\end{example}

Define $G \ce \Hom_\ZZ(\Cl(X), \CC^*)$. 
If $X$ is simplicial, as in the weighted projective example above, then the \defi{canonical stack} $\sX_{\can}:=[\Spec(S) - V(B) / G]$ has ``better'' properties with regards to subschemes and sheaves~\cite{BorisovChenSmith2004,GS}.  
Specifically, $\sX_{\can}$ is a global quotient stack whose coarse moduli space is $X$; see \Cref{defn:coarseModuli}. 
The global quotient property allows us to totally classify  sheaves on the stack; they are simply multigraded modules on $S$, considered up to $B$-torsion.  Similarly, ideal sheaves correspond to multigraded ideals, considered up to $B$-torsion; see \Cref{prop:equivSheaves-gradedMods-final}. 
The coarse moduli property then implies that this classification is very closely related to a similar classification for $X$.

The use of $\checkS$ extends this entire picture to the nonsimplicial case. 
Consider the group $\checkG\ce\Hom_\ZZ(\QCart(X),\CC^*)$. 
The $\QQ$-Cartier Cox ring $\checkS$ of $X$ determines a quotient stack, called the \defi{$\QQ$-Cartier stack} of $X$, defined by  
\[
\sXCar \coloneq [\Spec(\checkS) - V(\checkB) / \checkG].
\]

\begin{mainthm}
\label{thm:coarseModuli}
If $X$ is a semiprojective toric variety, not necessarily simplicial, then the following hold. 
\begin{enumerate}
\item The $\QQ$-Cartier stack $\sXCar$ of $X$ is a toric\footnote{A purely toric construction of $\sXCar$ is given in \Cref{sec:construction}.} DM stack with coarse moduli space $X$. 
\item The stack $\sXCar$ is universal for the property that the pullback of any $\QQ$-Cartier divisor is Cartier. 
\item The stack $\sXCar$ is smooth if and only if $X$ is simplicial.
\end{enumerate}
\end{mainthm}

We show in Section~\ref{sec:coarseModuli} that 
\Cref{thm:coarseModuli} implies 
\Cref{thm:nullstellensatz-var}
this consequence. We also provide a proof of \Cref{thm:nullstellensatz-var} in Section~\ref{sec:varProof} that is independent of stacky language. 

In addition, \Cref{thm:coarseModuli} allows one to classify quasi-coherent sheaves (and ideal sheaves) on $\sXCar$ and $X$ in concrete terms.  
In Cox's paper, this is done via \cite{cox:homog}*{Theorem 3.7}, where he shows that every quasi-coherent sheaf on a simplicial toric variety $X$ has the form $\widetilde{M}$ and that $\widetilde{M}=0$ if and only if the following somewhat technical statement holds: for every homogeneous element $m \in M_\alpha$ where $ \deg(m)=\alpha$ lies in $\Pic(X)$, the element $m$ is annihilated by a power of $B$. 
We extend to the nonsimplicial case in a couple of ways.  First, using $\checkS$ and the $\QQ$-Cartier stack, we prove a similar result in \Cref{thm:nullstellensatz-stack}. 
Second, we reinterpret Cox's technical condition in terms of a third Cox ring/stack perspective, defined below; a precise analogue of Cox's theorem appears in \Cref{cor:CoxQcoh}. 

\begin{defn}\label{defn:PicCox}
Define $S_{\Pic} \ce \bigoplus_{d \in \Pic(X)} H^0(X,\cO_X(d))$ with irrelevant ideal $B_{\Pic} \ce B\cap S_{\Pic}$, group $G_{\Pic}:=\Hom_\ZZ(\Pic(X), \CC^*)$, and stack $\sX_{\Pic}:=[\Spec(S_{\Pic})- V(B_{\Pic})/G_{\Pic}]$. 
\end{defn}

In the simplicial case, $\checkS$ is the Cox ring itself and $S_{\Pic}$ takes the flavor of a Veronese subring of $\checkS$; see Section~\ref{sec:picard}. 
The relationship between the three various Cox rings of the toric variety $X$ is
$
S_{\Pic} \subseteq \checkS \subseteq S
$, 
induced by the corresponding inclusions among Cartier divisors, $\QQ$-Cartier divisors, and Weil divisors.  The relationship among the stacks and $X$ is a sequence of morphisms:
\[
\sXcan \xrightarrow{\nu} \sXCar \xrightarrow{\pi} \sX_{\Pic} \xrightarrow{\psi} X.
\]
We will show in \Cref{thm:Pic-iso} that $\psi$ is an isomorphism and thus \Cref{thm:coarseModuli} shows that $\pi$ is always a coarse moduli map. 
Furthermore, $X$ is simplicial if and only if $\nu$ is the identity map \cite{CLSToricVarieties}*{Proposition 4.2.7} and $X$ is smooth if and only if $\pi$ is the identity map \cite{CLSToricVarieties}*{Proposition 4.2.6}. 
The properties of $\nu$, $\pi$, and $\psi$ provide a geometric explanation for many results in this paper.  For instance, the failure of $\nu$ to be a good moduli morphism in the nonsimplicial case ``explains" the failures of \Cref{thm:cox} in the nonsimplicial case, whereas the properties of $\pi$ explain why passing to $\checkS$ and $\sXCar$ remedy these issues. 

\medskip

Our primary motivation for rectifying problems arising in the study of nonsimplicial toric varieties is that they appear abundantly and naturally when considering birational toric morphisms and GKZ theory.
Recall that the \defi{GKZ fan} of a semiprojective toric variety $X$, denoted $\Sigma_{GKZ}(X)$, pa\-ram\-e\-trizes the birational geometry of $X$ in the following sense: to each cone $\Gamma = (\overline \Sigma_\Gamma, I_\Gamma)$ in the GKZ fan, there is an associated toric variety $X_\Gamma \ce X_{\Sigma_\Gamma}$ (where $\Sigma_\Gamma$ is the fan associated to the generalized fan $\overline \Sigma_\Gamma$), and for each containment $\Gamma' \subseteq \Gamma$, there is a morphism of toric varieties $X_{\Gamma}\to X_{\Gamma'}$. 
That is, this association between cones in the GKZ fan and toric varieties is functorial.  See \cite{CLSToricVarieties}*{Chapter 14} for a review.

This functoriality fails wildly if these toric varieties are replaced with their associated canonical stacks (see \Cref{ex:P113-H3-GKZcan}). 
It still fails if the canonical stack is replaced with $\sXCar$, 
when $\overline{\Sigma}_\Gamma$ does not use all rays in the original GKZ problem. 
We restore this functoriality by defining a minor variant $\sX_{\Gamma}$ of $\sXCar$ that incorporates 
the data of the extraneous rays; see \Cref{defn:goodStacky}.  
The resulting theorem can be viewed as a lifting of GKZ theory to stacks.

\begin{mainthm}
\label{thm:mainGKZstacky}
For each cone $\Gamma$ in $\Sigma_{GKZ}$,  there is a toric DM stack $\sXGamma$ with a coarse moduli space $X_\Gamma$ such that, for any subcone $\Gamma' \subseteq \Gamma$, there is a functorial morphism that commutes with the usual functoriality of GKZ theory:
\[
\begin{tikzcd}
\cX_{\Gamma} \ar[r] \ar[d] & \cX_{\Gamma'} \ar[d] \\
X_{\Gamma} \ar[r] & X_{\Gamma'}.
\end{tikzcd}
\]
That is, there is a diagram $\cF$ from the poset category of cones in $\Sigma_{GKZ}$  to DM stacks and a natural transformation of diagrams 
\[
\begin{tikzcd}
\Sigma_{\text{GKZ}} \ar[r, "{\cF}"] \ar[dr, "F"'] & \{\text{DM stacks}\} \ar[d, "\text{coarse moduli}"] \\
& \{\text{Varieties}\},
\end{tikzcd}
\]
where $F$ is the usual diagram in varieties.
Further, $\sX_{\Gamma}$ lies in between $X_{\Gamma}$ and the $\QQ$-Cartier stack $\sXGC$ of $X_{\Gamma}$ via finite maps
$\sXGC \to \sXGamma \to X_{\Gamma}$, 
and when $\Gamma$ lies within the moving cone, there is an equality 
$\sX_{\Gamma}=\sXGC$.  
\end{mainthm}

Finally, in \cite{ballard2024king}, a smooth toric DM stack $\widetilde \sX$ is constructed that lies over all of the $\sXGamma$ for each chamber $\Gamma$ in $\Sigma_{GKZ}$, which played a key role in the main results.  Our stacky GKZ diagram is readily compliant with this construction.  
That is, we show in \Cref{prop:DCox} that  functoriality extends to the common stacky refinement $\tsX$ defined in \cite{ballard2024king}*{Section~3.1}.

\smallskip
\subsection*{Outline}
\Cref{sec:varProof} provides a first proof of \Cref{thm:nullstellensatz-var}. 
Turning to stacks, a short introduction and motivation is available in \Cref{sec:background}, followed by an explicit alternate construction of the $\QQ$-Cartier stack $\sXCar$ in \Cref{sec:construction}. 
\Cref{sec:coarseModuli} proves \Cref{thm:coarseModuli}, as well as a stacky proof of \Cref{thm:nullstellensatz-var}. 
Another coarse moduli result is provided in \Cref{sec:picard} through the restriction of the Cox ring to its Picard degrees. 
Finally, \Cref{sec:GKZ} addresses GKZ theory, including the proof of \Cref{thm:mainGKZstacky}. 

\section{First proof of toric Nullstellensatz}
\label{sec:varProof}

In this section, we provide a direct proof of \Cref{thm:nullstellensatz-var}.  
A proof that uses toric stacks will appear in \Cref{sec:coarseModuli}. 

In \cite{cox:homog} and \cite{CLSToricVarieties}*{Chapter 5}, the Cox ring and its relationship to an ideal-variety and module-sheaf correspondence for toric varieties are developed.  
Those results rely on expressions of open subvarieties of $X=X_\Sigma$ as affine varieties with respect to localizations at monomials.  Namely, for any maximal cone $\sigma\in\Sigma_{\max}$, there is an affine toric variety $U_{\sigma}$ (defined by the subfan of $\Sigma$ induced by $\sigma$) and a distinguished squarefree monomial $x_{\sigma^c}\ce \prod_{\rho\notin\sigma(1)}x_\rho$, so that $\Spec(S[x_{\sigma^c}^{-1}]_0) = U_\sigma$~\cite{cox:homog}*{Lemma 2.2}.

Since the monomial $x_{\sigma^c}$ may not lie in $\checkS$, the localization $\checkS[x_{\sigma^c}^{-1}]_0$ might not make sense.  
However, localizing at a monomial like $x_0x_1$ is the same as localizing at $x_0^ax_1^b$ for any $a,b>0$.  
Thus, it is enough to find a monomial $\check{x}_{\sigma^c}$ which lies in $\checkS$ and where $x_i$ divides $x_{\sigma^c}$ if and only if $x_i$ divides $\check{x}_{\sigma^c}$, i.e. the squarefree part of  $\check{x}_{\sigma^c}$ equals $x_{\sigma^c}$.
The following lemma shows that this is always possible, allowing us to provide analogues of Cox's constructions for the $\QQ$-Cartier Cox ring. 

\begin{lemma}\label{lem:checkxsigma}
If $X$ is a semiprojective toric variety with fan $\Sigma$ and $\checkS$ as in \Cref{defn:intCoxAndCartierStack}, then there exists a monomial
$
\check{x}_{\sigma^c} = \prod_{\rho \notin \sigma(1)} x_{\rho}^{a_\rho}
$
satisfying the following properties: 
\begin{enumerate}
    \item  $a_\rho > 0$ for all $\rho$, and
    \item  $\check{x}_{\sigma^c}$ lies in the subring $\checkS$.
\end{enumerate}
Moreover, for any such choice, 
$\checkS[\check{x}_{\sigma^c}^{-1}]_0 = S[x_{\sigma^c}^{-1}]_0$, 
and hence 
\[
\Spec\left( \checkS[\check{x}_{\sigma^c}^{-1}]_0\right)
=
\Spec\left(S[x_{\sigma^c}^{-1}]_0 \right)
= U_{\sigma}.
\]
\end{lemma}
\begin{proof}
By definition of the irrelevant ideal $B$ of $S$ in \cite{CLSToricVarieties}*{Section~5.1}, the monomial $x_{\sigma^c}\in B$.  Write $\deg_{\RR}(x_{\rho})$ for the image of $\deg(x_{\rho})$ under the natural map $\Cl(X)\to \Cl(X)\otimes_{\ZZ}\RR$.  By \cite{CLSToricVarieties}*{Theorem 15.1.10 and Proposition 15.2.1}, the positive cone spanned by the vectors $\{\deg_{\RR}(x_{\rho}) \mid \rho \notin \sigma(1)\}$ will contain the nef cone of $X$ in $\Cl(X)\otimes_{\ZZ}\RR$.
In particular, this implies that there is some positive combination $\sum_{\rho\notin \sigma(1)} a_\rho \deg_{\RR}(x_{\rho})$ with $a_\rho \in \ZZ_{>0}$ that lies in the ample cone of $X$.  
Choose any such sequence of positive integers $a_\rho$ for $\rho \notin \sigma(1)$ and define $\check{x}_{\sigma^c}:= \prod_{\rho\notin \sigma(1)} x_{\rho}^{a_{\rho}}$.  Since every ample divisor is Cartier (and thus $\QQ$-Cartier), it follows that $\check{x}_{\sigma^c}$ satisfies the necessary conditions.

By property (1), $S[\check{x}_{\sigma^c}^{-1}] = S[x_{\sigma^c}^{-1}]$, so these rings agree in degree $0$.  
On the other hand, an element of $S[\check{x}_{\sigma^c}^{-1}]_0$ is a fraction $\frac{f}{\check{x}_{\sigma^c}^\ell}$ of total degree $0$, and therefore since $\check{x}_{\sigma^c}^\ell$ has $\QQ$-Cartier degree, so does $f$, i.e., $f\in \checkS$. 
Therefore $S[\check{x}_{\sigma^c}^{-1}]_0 = \checkS[\check{x}_{\sigma^c}^{-1}]_0$, completing the proof.
\end{proof}

\begin{defn}
For this section,  fix a choice of $\check{x}_{\sigma^c}$ as in \Cref{lem:checkxsigma} and set 
$\checkS_{\sigma}:=\checkS[\check{x}_{\sigma^c}^{-1}]$.    
\end{defn}

Now Cox's module-sheaf correspondence extends as follows.

\begin{defn}\label{def:barNotation}
Let $M$ be a (finitely generated) graded $\checkS$-module.
There is an associated quasi-coherent (coherent) sheaf $M|_X$
 defined by $M|_X(U_\sigma) \ce M[\check{x}_{\sigma^c}^{-1}]_0$ for each $\sigma\in\Sigma_{\max}$.  In particular, for a homogeneous ideal $I \subseteq \checkS$, denote by $I|_X \subseteq \cO_X$ the corresponding ideal sheaf and by $V_X(I)$ the corresponding closed subset.
\end{defn}

The sheaf in \Cref{def:barNotation} is often denoted $\widetilde{M}$, but we use $M|_X$ to avoid potential confusion when a single module $M$ determines a sheaf on different varieties (or later, stacks) with the same $\QQ$-Cartier Cox ring. 
Observe that if $M'$ is a graded $S$-module and $M=\bigoplus_{d \in \QCart(X)} M'_d$ is the $\checkS$-module obtained by restricting degrees, then a minor variant of the proof of \Cref{lem:checkxsigma} shows that $M'|_X$ and $M|_X$ induce the same sheaves on $X$ (because they have the same sections on $U_{\sigma}$ for any $\sigma$).

\medskip 
    
If $\checkG\ce \Hom(\QCart(X), \CC^*)$, then $\QCart(X)$ can be viewed as the character group of $\checkG$. 
Viewing $\deg{s}$ as a character of $\checkG$, the $\QCart(X)$-grading on $\checkS$ induces an action of $\checkG$ on $\checkS$ via $g\cdot s \ce (\deg{s})(g)s$, which induces a $\checkG$-action on $\Spec(\checkS)$ and $\Spec(\checkS) - V(B)$. 
We will show that $X$ is a good geometric quotient (see \Cref{defn:geometricQuotient}) for this action. 
The proof will require the following standard fact.

\begin{lemma}\label{lem:invariants}
If $\sigma\in \Sigma_{\max}$, then the invariant ring $\checkS_{\sigma}^{\checkG}$ equals $(\checkS_{\sigma})_0$.
\end{lemma}
\begin{proof}
This is immediate from the definition of the $\checkG$-action.  Namely, if $s/s'\in \checkS_{\sigma}$ with $s,s'$ homogeneous, then $g\cdot (s/s') = (\deg{s})(g)s / (\deg{s'})(g)s'$, which equals $s/s'$ for all $g$ if and only if $\deg(s)=\deg(s')$.  
\end{proof}

\begin{defn}\label{def:goodCatQuot}
A map $\pi\colon Y \to X$ is a \defi{good categorical quotient} of $Y$ by $G$ if
\begin{enumerate}
    \item The morphism $\pi$ is $G$-invariant, surjective, and affine.
    \item  The natural homomorphism of $\cO_Y$-modules 
    $\cO_Y \xrightarrow{\sim} (\pi_* \cO_X)^G$ 
    is an isomorphism. \item  If $W \subseteq X$ is a $G$-invariant closed subset, then its image $\pi(W)$ is closed in $Y$.
    \item If $W_1, W_2 \subseteq X$ are disjoint $G$-invariant closed subsets, then their images $\pi(W_1)$ and $\pi(W_2)$ are disjoint in $Y$.
    \item $\pi$ is universal among $G$-invariant morphisms to varieties.   
    That is, for any variety $Z$ and any $G$-invariant morphism $f\colon Y \to Z$, there exists a unique morphism $g\colon X \to Z$ such that $f = g \circ \pi$:
\[
\begin{tikzcd}[row sep=large, column sep=large]
Y \arrow[r, "\pi"] \arrow[dr, "f"'] & X \arrow[d, "g", dashed] \\
& Z.
\end{tikzcd}
\]
\end{enumerate}
\end{defn}

\begin{prop}\label{prop:goodCatQuot}
If $X$ is a semiprojective toric variety, not necessarily simplicial, and $Y = \Spec(\checkS)-V(\checkB)$, then the natural map $\pi\colon Y\to X$ 
is a good categorical quotient.
\end{prop}
\begin{proof}
By \cite{CLSToricVarieties}*{Proposition 5.0.12}, it suffices to show that for each open $U_\sigma\subseteq X$, the induced $\pi^{-1}(U_\sigma)\to U_\sigma$ is a good categorical quotient. 
This follows immediately by combining \Cref{lem:invariants} and \cite{CLSToricVarieties}*{Proposition 5.0.9}.
\end{proof}

\begin{defn} \label{defn:geometricQuotient} For an algebraic group $G$ acting on a variety $Y$, 
a map $Y \to X$ is a \defi{geometric quotient} of $Y$ by $G$ if it is a good categorical quotient and $G$ acts transitively on the fibers.
\end{defn}

\begin{theorem}\label{thm:orbits}
If $X$ is a semiprojective toric variety, not necessarily simplicial, and $Y = \Spec(\checkS)-V(\checkB)$, then $Y \to X$ is a geometric quotient. 
\end{theorem}
\begin{proof}
By \Cref{prop:goodCatQuot}, $Y \to X$ is a good categorical quotient. 
It remains to show that $G$ acts transitively on the fibers of $Y \to X$.  
This can be checked locally on the base, and so $X$ can be replaced by $U_\sigma=\Spec\left((\checkS_{\sigma})_0\right)$ and $Y$ by $\Spec \checkS_\sigma$. 

The point $p\in U_\sigma$ corresponds to a maximal ideal $\fp\subseteq (\checkS_\sigma)_0$, and the reduced fiber $\pi^{-1}(p)$ has coordinate ring $\checkS_\sigma\otimes_{(\checkS_\sigma)_0} (\checkS_\sigma)_0/\fp$. A point $\tp\in\pi^{-1}(p)$ corresponds to a surjective map 
\[
\ev_{\tp}\colon \checkS_\sigma \otimes_{(\checkS_\sigma)_0} (\checkS_\sigma)_0/\fp \to \CC.
\]
As a scheme, the group $\checkG$ is $\Spec(\CC[\QCart(X)])$,  
where $\chi^d\in \CC[\QCart(X)]$ denotes the monomial corresponding to $d\in \QCart(X)$. 
The $\checkG$-action on the fiber
\begin{align*}
\checkG \times \widetilde{p} & \to \pi^{-1}(p) \\
g & \mapsto g \cdot \widetilde{p}
\end{align*}
is given explicitly by the following ring map, 
\begin{align*}
\zeta\colon \checkS_\sigma \otimes_{(\checkS_\sigma)_0} (\checkS_\sigma)_0/\fp &
\to \CC[\QCart(X)] \\
f \otimes v & \mapsto \ev_{\tp}(fv)\chi^{\deg f}.
\end{align*}
Hence it is enough to show that this map is finite and injective, so that the map $\checkG \to \pi^{-1}(p)$ is surjective by lying over (see, for example, \cite{eisenbud}*{Proposition 4.15}).  

For finiteness of $\zeta$, any choice of $\check{x}_{\sigma^c}$ from \Cref{lem:checkxsigma}  is a unit in $\checkS_\sigma$, so $\ev_{\tp} (\check{x}_{\sigma^c}\otimes 1) = a \neq 0$ and 
$\zeta(\check{x}_{\sigma^c}\otimes 1) =  a \check{x}_{\sigma^c}$. 
The proof of \Cref{lem:checkxsigma} shows that $\{\deg(x_\rho)\mid\rho\in\sigma^c\}\cap\Amp(X)$ contains a finite index submonoid $Q_\sigma$ of $\Amp(X)\cap\Cl(X)$. 
Varying through all choices of $\check{x}_{\sigma^c}$ implies that the image of $\zeta$ contains 
$\CC[Q_\sigma]$. 
Since $\check{x}_{\sigma^c}$ is a unit in the source of $\zeta$, the image of $\zeta$ contains $\CC[\ZZ Q_\sigma]$. 
Since $\ZZ Q_\sigma$ is a finite index subgroup of $\ZZ(\Amp(X)\cap\Cl(X))$, which is turn a finite index subgroup of $\QCart(X)$, the map $\zeta$ must be finite.

To show that $\zeta$ is injective, let $\tfp=\ker(\ev_{\tp})$ be the maximal ideal corresponding to $\tp$. 
We first claim that $\tfp$ contains no homogeneous elements. To see this, suppose to the contrary that $f\otimes 1\in\tfp$ is homogeneous of degree $d\in\QCart(X)$. 
Then there exist $k\in\NN$ so that $kd\in\Pic(X)$ and a monomial $m\in (\checkS_\sigma)_{kd}$ of the form from \Cref{lem:checkxsigma}, so that it is a unit. 
This implies that $\frac{f^k}{m}\in (\checkS_\sigma)_0/\fp\cong \CC$, so there exists $\lambda\in\CC$ with $\frac{f^k}{m}=\lambda$. 
Now 
\begin{align}\label{eq:symbolPush}
f^k\otimes 1 = m \cdot \frac{f^k}{m}\otimes 1 = m\otimes \frac{f^k}{m} = m\otimes \lambda = \lambda m \otimes 1. 
\end{align}
Since $m$ is a unit in $\checkS_\sigma$, $\ev_{\tp}(m\otimes 1)$ is nonzero, and so $\ev_{\tp}(f^k\otimes 1) = \lambda\ev_{\tp}(m\otimes 1)$ is zero if and only if $\lambda=0$. 
By \eqref{eq:symbolPush}, $\lambda=0$ exactly when $f^k\otimes 1=0$, and since $\tfp$ is prime, this occurs if and only if $f\otimes 1=0$, which establishes the claim. 

Finally, consider a nonzero element of the form $\sum_d h_d\otimes 1$, where $h_d$ is of degree $d$. Then by the computations in the previous paragraph, 
\[
\zeta\left(\sum_d h_d\otimes 1\right) 
= \sum_d \ev_{\tp}(h_d)\chi^d, 
\]
which is nonzero since some $h_d$ is nonzero. Thus $\zeta$ is injective, as desired. 
\end{proof}

\begin{prop} \label{prop:geometricquotientClosedSub}
Let $X$ and $Y$ be algebraic varieties, where $Y$ has an action by an algebraic group $H$. If $\pi\colon Y \to X$ is a geometric quotient, then there is a bijection between closed subspaces of $X$ and $H$-invariant closed subspaces of $Y$.
\end{prop}
\begin{proof}
By definition, a geometric quotient induces a bijection
\vspace*{-1mm}
\[
\begin{tikzcd}
\left\{ H\text{-orbits of }Y \right\} \ar[r, bend left, "\pi"] & \left\{ \text{points of }X \right\} \ar[l, bend left, "\pi^{-1}"].
\end{tikzcd}
\]
\vspace*{-1mm}
Since an $H$-invariant closed subspace of $Z \subseteq Y$ is a union of $H$-orbits, $\pi^{-1}(\pi(Z)) = Z$.  
This induces the desired bijection
\vspace*{-4mm}
\[
\begin{tikzcd}
\left\{ H\text{-invariant subspaces of }Y \right\} \ar[r, bend left, "\pi"] & \left\{ \text{subspaces of }X \right\} \ar[l, bend left, "\pi^{-1}"].
\end{tikzcd}
\qedhere
\]
\end{proof}

We are now prepared to give a non-stacky proof of \Cref{thm:nullstellensatz-var}. 

\begin{thmA}[Toric Ideal-Variety Correspondence]
If $X$ is a semiprojective toric variety, not necessarily simplicial, then there is a bijective correspondence between closed subvarieties of $X$ and radical homogeneous ideals $I \subseteq \checkB \subseteq \checkS$.
\end{thmA}
\begin{proof}[Proof of \Cref{thm:nullstellensatz-var}]
By \Cref{thm:orbits}, $Y \to X$ is a geometric quotient.  Hence, by \Cref{prop:geometricquotientClosedSub}, there is a bijection between $\checkG$-invariant closed subspaces of $Y$ and $X$.  Thus $Y = \Spec \checkS - V(\checkB)$. 
Closed subspaces of $Y$ are of the form $Z - (Z\cap V(\checkB))$, where $Z$ is a closed subspace of $\Spec(\checkS$). 
There is a bijection between closed subspaces of $Y$ and closed subspaces of $\Spec \checkS$ that contain $V(\checkB)$ given by $Z - (Z\cap V(\checkB))\mapsto Z\cup V(\checkB)$. 
It follows that closed subspaces of $Y$ are in bijection with radical ideals of the form $I \subseteq B$.
Asking that $V(I)$ is $\checkG$-invariant corresponds algebraically to $I$ being homogeneous.
\end{proof}

A more modern perspective on quotients is to utilize the language of stacks, where Mumford's original definition of geometric quotients corresponds to the notion of coarse moduli map, as seen in \Cref{thm:coarseModuli}.  
This connection is discussed in Section~\ref{sec:background}.

To see various alternate ways to describe the closed subvarieties of $X$ in terms of radical ideals, recall the following notation. 

\begin{notation}\label{not:satutation}
Recall that given ideals $I,J$ in a ring $R$, the \defi{$J$-saturation} of $I$ is 
\[
I:J^\infty \ce \{r\in R\mid \text{there exists } s\in J \text{ such that } s^Nr\in I \text{ some } N\geq 0\}. 
\] 
\end{notation}

\begin{prop}\label{prop:alternateNullstellensatz}
There are natural bijections between each of the following sets:
\begin{enumerate}
	\item  Closed subsets of $X$;
	\item  Radical homogeneous ideals $I\subseteq \checkB \subseteq \checkS$;
	\item  $\checkB$-saturated, radical homogeneous ideals in $\checkS$;
	\item  $B_{\Pic}$-saturated, radical ideals in $S_{\Pic}$;
	\item  Radical, homogeneous ideals $I' \subseteq B_{\Pic} \subseteq S_{\Pic}$; and 
    \item  $\checkG$-invariant subspaces of $\Spec(\checkS)-V(\checkB)$.
\end{enumerate}
\end{prop}
Before turning to the proof, we record a more general lemma.
\begin{lemma}\label{lem:veronese}
If $H$ is a finitely generated group with finite index subgroup $K$ and $R=\bigoplus_{h\in H} R_h$ is an $H$-graded ring with $K$-graded subring  $R'=\bigoplus_{k\in K} R_{k}$, 
then there is a natural bijection as follows: 
\begin{align*}
\left\{\text{radical homogeneous ideals in } R \right\}
&\leftrightsquigarrow 
\left\{ 
\begin{matrix}\text{radical homogeneous ideals in } R' 
\end{matrix}
\right\}
\\
\Psi\colon \,\,\, \qquad I &\mapsto I\cap R', \text{ and}\\
\Phi\colon \ \ \,  \sqrt{JR} &\mapsfrom J.
\end{align*}
\end{lemma}
\begin{proof}
We first show that for a radical homogeneous $R$-ideal $I$, $I\subseteq \Phi \circ \Psi(I)$.  
If $f\in I$ is homogeneous of degree $h\in H$, then since $K\subseteq H$ is of finite index, $f^n \in I \cap R'=\Psi(I)$ for some $n$.  
Thus $f^n \in \Psi(I) R$, so $f\in \sqrt{ \Psi(I) R'} = \Phi\circ  \Psi(I)$, establishing the claim.  

Conversely, we show that $\Phi \circ \Psi(I) \subseteq I$.  
If $f\in\Phi \circ \Psi(I)= \sqrt{ I \cap R'}$, then $f^n \in I\cap R'$ for some $n$. 
This means that $f^n$ is in $I$, and $I$ is radical, so  $f\in I$, as desired. 

Next, we show that for radical homogeneous $R'$-ideal $J$, $J \subseteq \Psi\circ \Phi(J)$.  
If $g\in J$, then $g\in (\Phi(J))=\sqrt{JR}$. 
Since $g\in R'$ to begin with, it follows that $g \in (\Psi(J)) \cap R'$, as desired. 

Conversely, we show that $\Psi\circ \Phi(J)\subseteq J$. 
If $g\in \Psi\circ \Phi(J)$, then $g\in (\Phi(J)) \cap R'=(\sqrt{JR})\cap R'$, so there exists $n$ such that $g^n \in JR$. 
Now write $g^n = \sum a_ig_i$ for some $a_i \in R$ and a choice of generators $\{g_i\}$ of $J\subseteq R'$.  
Since the degrees of $g^n$ and each $g_i$ belong to$K$, it follows that the degree of each $a_i$ must also belong to $K$.  
Thus $g^n \in J$, and since $J$ was radical, $g\in J$, establishing the final claim.
\end{proof}

\begin{proof}[Proof of \Cref{prop:alternateNullstellensatz}]
The equivalence of (1) and (2) is \Cref{thm:nullstellensatz-var}. 
Lemma~\ref{lem:veronese} implies the equivalence of (2) and (4), as well as that of (3) and (5). 
The bijection from (2) to (3) is given by $I\mapsto I:\checkB^\infty$, where the converse is given by intersecting with $\checkB$; the bijection between (4) and (5) is similar. 
Finally, the equivalence of (1) and (6) follows from \Cref{thm:orbits} and \Cref{prop:geometricquotientClosedSub}.
\end{proof}

\section{Stacks, graded modules, and quotients}
\label{sec:background}

\Cref{prop:alternateNullstellensatz} demonstrates that closed subvarieties of $X$ can be  described equally well in terms of radical ideals in $\checkS$ or $S_{\Pic}$.  
The comparison with subschemes of $X$, or more generally, quasi-coherent sheaves on $X$, is more subtle, and we will use the language of stacks to clarify these results.  
In this section, we motivate the presence of toric stacks in this article by explaining how their quasi-coherent sheaves are more naturally related to graded modules on $\checkS$ (or $S_{\Pic})$.  Then we obtain a toric Nullstellensatz result for $\sXCar$, and provide a proof for \Cref{thm:coarseModuli}.(1). 

To begin, note that one key property of quotient stacks (all stacks in this paper are quotient stacks), which is that quasi-coherent sheaves on a quotient $[Y/G]$ are naturally equivalent to $G$-equivariant sheaves on $Y$.  
This allows for concrete descriptions of sheaves and subschemes in terms of ideals and graded modules on $\checkS$ (or $S_{\Pic})$, analogous to results of Cox for the simplicial case.  
For background on stacks, there are several major references such as~\cite{olsson, vistoli, stacks-project, alper}, though we will utilize only a few specialized results from this literature. 

Let $\QCoh(-)$ denote a category of quasi-coherent sheaves. 
Further, write $\Mod_{\Cl(X)}(S)$ for the category of $\Cl(X)$-graded $S$-modules and $\Mod_{\Cl(X),B}(S)$ for the full subcategory of modules $M$, where each element $m\in M$ is annihilated by some power of $B$.   

When $X=\PP^n$, it is well known that the category $\QCoh(\PP^n)$ is the Verdier quotient $\Mod_{\ZZ}(S) / \Mod_{\ZZ,\fm}(S)$, where $\fm$ denotes the graded maximal ideal of $S$.  
In short, studying quasi-coherent sheaves on $\PP^n$ is equivalent to studying $\ZZ$-graded $S$-modules, modulo $\fm$-torsion $S$-modules.  If $X$ is simplicial, but not smooth, then the corresponding result is slightly more complicated.  In this case,  the category of sheaves on the canonical stack $\sX_{\can}$ is more closely connected to multigraded commutative algebra than the category of sheaves on the variety $X$.
The following result is essentially a stacky rephrasement of Cox's results about sheaves on simplicial toric varieties~\cite{cox:homog}.

\begin{prop}[Module-Sheaf Correspondence, Simplicial Case]\label{prop:moduleSheafCorrespondence}
If $X$ is a simplicial toric variety, then there are equivalences of categories: 
\[
\QCoh(\sX_{\can}) 
= \frac{ \Mod_{\Cl(X)}(S)}{\Mod_{\Cl(X),B}(S)} 
\quad \text{ and } \quad 
\QCoh(X) = \frac{ \Mod_{\Pic}(S_{\Pic})}{\Mod_{\Pic,B_{\Pic}}(S_{\Pic})}.
\]
\end{prop}

Thus, if one is interested in understanding modules over the multigraded polynomial ring $S$, then $\sX_{\can}$ is the ``right'' geometric object to consider, whereas $X$ itself is more closely associated to the study of modules over the subring $S_{\Pic}$.

\begin{example}\label{ex:P113-3}
Let $X=\PP_v(1, 1, 3)$ be a weighted projective variety. 
In this case, its Cox ring is $S=\CC[x,y,z]$ with degrees $1, 1, 3$, while $S_{\Pic}$ is the Veronese subring consisting of all degrees that are multiples of $3$, so that $S_{\Pic} = \CC[x^3,x^2y,xy^2,y^3,z]$. 
\Cref{ex:P113-1} shows how multiple homogeneous ideals (namely, $I_1$, $I_2$, and $I_3$ in that example) all yield the same ideal sheaf on $\PP_v(1, 1, 3)$. This is not the case on the associated canonical stack, which is the weighted projective stack $\sX_{\can} = \PP(1,1,3)$, since \Cref{prop:moduleSheafCorrespondence} implies that the ideals $I_1$, $I_2$, and $I_3$ each determine distinct sheaves on the stack $\PP(1,1,3)$.  
Thus the stack $\PP(1,1,3)$ is more closely tied to the regular ring $S$, where the variety $\PP_v(1,1,3)$ is more closely tied to the singular subring $S_{\Pic}$.  
\end{example}

The first goal of this section is to explain how the $\QQ$-Cartier Cox ring $\checkS$ allows us to generalize \Cref{prop:moduleSheafCorrespondence} to the nonsimplicial toric case.  

\begin{theorem}[Module-Sheaf Correspondence, General Case]\label{thm:nonsimplicialModSheaf}
If $X$ is a semiprojective toric variety, not necessarily simplicial, then there are equivalences of categories: 
\[
\QCoh(\sX_{\QCart}) 
= \frac{ \Mod_{\QCart(X)}(\checkS)}{\Mod_{\QCart(X),\checkB}(\checkS)} 
\quad \text{ and } \quad 
\QCoh(X) 
= \frac{ \Mod_{\Pic(X)}(S_{\Pic})}{\Mod_{\Pic(X),B_{\Pic}}(S_{\Pic})}.
\]
\end{theorem}

Similar statements hold for the corresponding derived categories as well. 
These results are all examples of more general statements about sheaves on a quasi-affine variety $Y$ of the form $\Spec(R)-V(I)$ with an action of a sufficiently nice algebraic group $H$ with character group $\wH\ce \Hom(H, \CC^*)$.
We now explain how an $H$-action on $\Spec(R)$ induces a $H$-action on $R$. 
Since $H$ is a diagonalizable group, $R$ decomposes as a direct sum over the character group $\wH$, which induces a $\wH$-grading on $R$. 
For example, if $H=(\CC^*)^2$ acts on $R = \CC[x,y]$ by $(\lambda, \mu)\cdot x = \lambda^1\mu^{-2}x$ and $(\lambda, \mu)\cdot y = \lambda^3\mu^{1}y$, then $R = \CC[x,y]$ inherits a $\ZZ^2$-grading with $\deg(x) = (1,-2)$ and $\deg(y)=(3,1)$.  
Any $H$-equivariant $R$-module $M$ also inherits a $\wH$-grading in a natural way, which yields the equivalence between $H$-equivariant $R$-modules and $\wH$-graded $R$-modules, which we make more precise in the next two results.

\begin{prop} 
\label{prop:diagonalizableAffineCase}
If $H$ is a diagonalizable algebraic group with character group $\widehat H$ and $R$ is a ring with an $H$-action, then there is an equivalence of categories
\[
 \QCoh_H(\Spec (R)) \cong \Mod_{\widehat H}(R), 
\]
where $\QCoh_H(\Spec (R))$ is the category of $H$-equivariant quasi-coherent sheaves on $\Spec (R)$.
\end{prop}
\begin{proof}
By definition, for a diagonalizable algebraic group, 
\[
\Gamma(H, \cO_H) = \CC[\widehat H] = \bigoplus_{\chi \in \widehat H} \CC. 
\]
Let $\cF$ be an $H$-equivariant quasi-coherent sheaf on $\Spec(R)$, meaning there is an isomorphism $\varphi\colon \sigma^*\cF\to p^*\cF$, where $\sigma, p\colon G \times Y \to Y$ 
denote the group action map and  projection onto the second factor, respectively. 
Setting $M= \Gamma(\Spec(R),\cF)$, $\Gamma(\varphi)$ induces a map 
\begin{align*}
\nu_M = \Gamma(\varphi)|_M\colon
M \subseteq M[\widehat{H}] 
&\to M[\widehat{H}] 
= \bigoplus_{\chi\in\widehat{H}} M\otimes \chi,
\\ 
m \qquad \qquad
&\mapsto \qquad \qquad
\textstyle\sum m_\chi\otimes \chi. 
\end{align*}
Thus $M$ has the structure of a $\widehat{H}$-graded $R$-module via 
$M = \bigoplus_{\chi \in \widehat H} \, \nu_M^{-1}(\chi \otimes M)$. 

Conversely, a $\widehat{H}$-graded $R$-module $M=\bigoplus_{\chi \in \widehat H} M_\chi$ induces a map 
$\nu_M\colon M[\widehat{H}]\to M[\widehat{H}]$ where a homogeneous $m\in M_\chi$ has $\nu_M(m\otimes 1) = m\otimes \chi$. Taking the associated sheaf then induces an element of 
$\QCoh_H(\Spec (R))$.  

Hence the usual equivalence $\Gamma$ takes $H$-equivariant sheaves to $\widehat{H}$-graded modules, and one checks that under this construction, $H$-equivariant morphisms of quasi-coherent sheaves on $\Spec(R)$ are taken exactly to degree zero morphisms of $\widehat{H}$-graded $R$-modules. 
\end{proof}

While not stated in precisely this form in our search of the literature, the following result is well known to experts and follows from other more general results about stacks, including \cite{stacks-project}*{06WT}, \cite{vistoli}*{Theorem 4.46},  \cite{AOV2008}*{Section~2.1}, or \cite{alper}*{Theorem 7.1.11}. 

\begin{prop} 
\label{prop:equivSheaves-gradedMods-final}
Let $H$ be a diagonalizable algebraic group acting on $\Spec(R)$, so that $\wH = \Hom(H,\CC^*)$, which induces an $\wH$-grading on $\Spec(R)$. 
If $J\subseteq R$ is an ideal that is homogeneous with respect to the $\wH$-grading, then 
\begin{enumerate}
\item The category of quasi-coherent sheaves on the global stack quotient $[\Spec(R)/H]$ is naturally equivalent to the category of $H$-equivariant $R$-modules, which is naturally equivalent to the category of $\wH$-graded $R$-modules.  
\item If $Y=\Spec(R) - V(J)$, then there are equivalences:
\[
\QCoh( [Y/H] ) 
\cong 
\QCoh_H(Y)
\cong
\Mod_{\wH}(R) / 
\Mod_{(\wH,I)}(R). 
\]
where $\QCoh_H(Y)$ is the category of $H$-equivariant quasi-coherent sheaves on $Y$.
\hfill$\square$
\end{enumerate}
\end{prop}

\begin{proof}
[Proofs of \Cref{prop:moduleSheafCorrespondence} and \Cref{thm:nonsimplicialModSheaf}] 
\Cref{prop:moduleSheafCorrespondence} is a special case of \Cref{thm:nonsimplicialModSheaf}. 
To prove \Cref{thm:nonsimplicialModSheaf}, since $\sX_{\QCart}$ is defined as $[\Spec(\checkS) - V(\checkB)/\checkG]$, so the first statement follows from substitution; further, the fact that $X \cong [\Spec(S_{\Pic}) - V(B_{\Pic})/G_{\Pic}]$ (see \Cref{thm:Pic-iso}) yields the second statement. 
\end{proof}

Corollaries of \Cref{prop:equivSheaves-gradedMods-final} are analogues of the Nullstellensatz results \cite{cox:homog}*{Proposition 3.5 and Theorem 3.7} and \cite{CLSToricVarieties}*{Propositions 5.3.9-10} that describe quasi-coherent sheaves and closed subschemes on $X$ and $\sX_{\QCart}$.  
We begin with a general statement. 

\begin{cor} \label{cor:stackyNullstellensatz}
Let $H$ be a diagonalizable algebraic group and $Y = \Spec (R) - V(J)$.
\begin{enumerate}[noitemsep]
\item 
Every ideal sheaf on the stack $[Y/H]$ has the form $I|_{[Y/H]}$ for some $I\subseteq R$, and $I|_{[Y/H]} = I'|_{[Y/H]}$ if and only if $I\colon J^\infty = I'\colon  J^\infty$.
Furthermore, $H$-invariant closed subschemes of $Y$ are in bijection with $J$-saturated homogeneous ideals $I \subseteq R$.
\item For every coherent sheaf $\cF$ on the stack $[Y/H]$, there is a finitely generated, multigraded $R$-module $M$ such that $\cF = M|_{[Y/H]}$.  Moreover, $M|_{[Y/H]}=0$ if and only if $J^\ell M=0$ for some $\ell$.
\end{enumerate}
\end{cor}
\begin{proof}
We begin with (2). 
As a direct consequence of \Cref{prop:equivSheaves-gradedMods-final}, every quasi-coherent sheaf $\cF$ is of the form $M|_{[Y/H]}$ for some graded $R$-module $M$. 
However, it requires an additional step to show that if $\cF$ is coherent, then $M$ can be chosen to be finitely generated.  While this is fairly standard (for example, \cite{CLSToricVarieties}*{Proposition 6.A.4} captures the main idea), we still provide a proof for the sake of completeness. 
If $J=\< a_1, \dots, a_r\>$, then the basic open affine sets $D(a_i)$ are an open cover of $Y$.  
Now choose generators $m_{i,j}\in \cF(D(a_i)) = M[a_i^{-1}]$.  
Then there is a sufficiently large integer $N$ such that for each pair $i,j$, $a_i^Nm_{i,j}$ belongs to $M$.  
Consider the submodule $M'$ of $M$ generated by all of the $a_i^Nm_{i,j}$ for all $i,j$.  
Since $M'|_{D(a_i)} = M|_{D(a_i)}$ for all $i$, it follows that $M'$ and $M$ induce the same sheaf on $U$, so that $M'$ is finitely generated.

By \Cref{prop:equivSheaves-gradedMods-final}, the module $M$ induces the zero sheaf on $Y$ if and only if $M$ lies in $\Mod_{(\wH,I)}(R)$, which is equivalent to every element of $M$ being annihilated by some power of $J$. 
Since $M$ is finitely generated, this is the same as $J^\ell M =0$ for some $\ell\gg 0$.

For (1), note that by \Cref{prop:equivSheaves-gradedMods-final}, every ideal sheaf has the form $I|_{[Y/H]}$ for some homogeneous ideal $I$. 
This implies that homogeneous ideals of $R$ correspond to subschemes of $Y$, up to primary components supported on $V(J)$. 
The operation of saturation with respect to $J$ removes these primary components, so that $J$-saturated ideals of $R$ correspond uniquely to subschemes of $[Y/H]$.
\end{proof}

As an immediate corollary of this fact, we obtain the following corollary for $\sX_{\QCart}$.

\begin{cor}\label{thm:nullstellensatz-stack} 
Let $X$ be a semiprojective toric variety, not necessarily simplicial, with $\QQ$-Cartier stack $\sXCar$. 
\begin{enumerate}[noitemsep]
\item 
Every ideal sheaf on $\sXCar$ has the form $I|_{\sXCar}$ for some $I\subseteq \checkS$, and $I|_{\sXCar} = I'|_{\sXCar}$ if and only if $I\colon\checkB^\infty = I'\colon\checkB^\infty$.
In particular, closed substacks of $\sXCar$ are in bijection with $\checkG$-invariant closed subschemes of $\Spec(\checkS) - V(\checkB)$ and $\checkB$-saturated homogeneous ideals $I \subseteq \checkS$.
\item For every coherent sheaf $\cF$ on  $\sXCar$, there is a finitely generated, multigraded $\checkS$-module $M$ such that $\cF = M|_{\sXCar}$.  Moreover, $M|_{\sXCar}=0$ if and only if $\checkB^\ell M=0$ for some $\ell$.
\hfill$\square$
\end{enumerate}
\end{cor}

We now translate the geometric quotient property from \Cref{thm:orbits} into the language of stacks, which requires the following definition.

\begin{defn}[\cite{olsson}*{Definition 11.1.1}]\label{defn:coarseModuli}
Let $\cX$ be an algebraic stack over $\Spec(\CC)$, 
and let $X$ be a variety. 
A morphism $\pi\colon \cX \to X$ is called a \defi{coarse moduli map} if it satisfies the following two properties:
\begin{enumerate}
\item The morphism $\pi$ induces a bijection between the set of isomorphism classes of objects in the groupoid $\cX(\CC)$ and the set of $\CC$-valued points of $X$:
\[
|\cX(\CC)| \xrightarrow{\sim} X(\CC).
\]
\item The map $\pi$ is universal among all morphisms from $\cX$ to varieties. That is, for any variety $Z$ over $R$ and any morphism $f\colon \cX \to Z$, there exists a unique morphism of varieties $g\colon X \to Z$ such that the following diagram commutes:
\[
\begin{tikzcd}[row sep=large, column sep=large]
\cX \arrow[r, "\pi"] \arrow[dr, "f"'] & X \arrow[d, "g", dashed] \\
& Z.
\end{tikzcd}
\]
\end{enumerate}
\end{defn}

\begin{cor}\label{cor:coarseModuli}
The map $\sX_{\QCart} \to X$ is a coarse moduli map.
\end{cor}
\begin{proof}
Given that $\sX_{\QCart} = \Spec(\checkS) - V(\checkB) \to X$ is a geometric quotient, by \Cref{thm:orbits}, the $\CC$-points of $\sX_{\QCart}$ correspond to $\checkG$-orbits, and these are precisely the points of $X$ by \cite{CLSToricVarieties}*{Proposition 5.0.8}. 
This yields condition (1) of Definition~\ref{defn:coarseModuli}.  
For condition (2), the universal property of the quotient stack $[\Spec(\checkS) - V(\checkB)  / \checkG]$ lines up precisely with the universal property of the good categorical quotient.
\end{proof}

While \Cref{cor:coarseModuli} implies  \Cref{thm:coarseModuli}.(1), an alternate proof of this fact that utilizes a toric stacky datum description of $\sX_{\QCart}$ can be found in Section~\ref{sec:coarseModuli}.

\Cref{prop:geometricquotientClosedSub} is sufficient to imply \Cref{thm:coarseModuli}.(1). 
Since the closed points of $\sXCar = [\Spec(\checkS) - V(\checkB) / \checkG]$ correspond to $\checkG$-orbits, the bijection on those points translates exactly to the property of being a geometric quotient by \cite{CLSToricVarieties}*{Proposition 5.0.8}.
Thus, for a good categorical quotient like $\pi\colon \sXCar\to X$, being a coarse moduli map is equivalent to being a geometric quotient. 

\begin{cor}\label{cor:coarseModuli-bijection}
If $Y = \Spec (R) - V(J)$, $[Y/H] \to X$ is a coarse moduli map, and $H$ is a diagonalizable algebraic group, then there is a bijection between closed subspaces of $X$ and $\wH$-homogeneous ideals $I \subseteq J \subseteq R$.
\end{cor}
\begin{proof}
Since $Y\to X$ is a coarse moduli map, the closed subspaces of $X$ are in bijection with $H$-invariant closed subspaces of $Y$. Further, closed subspaces of $Y$ are in bijection with $I \subseteq J \subseteq R$ by \Cref{cor:stackyNullstellensatz}.  
\end{proof}

\section{The \texorpdfstring{$\QQ$}{Q}-Cartier stack via toric stack datum} 
\label{sec:construction}

Just as toric varieties are generally described in terms of fans, toric stacks can be described in terms of toric stack data. 
In this section, we recall the notion of toric stack data and morphisms between them, and we describe $\sX_{\QCart}$ in terms of explicit toric stack data. 
This provides a more combinatorial perspective on $\sX_{\QCart}$. 
In Section~\ref{sec:coarseModuli}, we will employ this description to prove \Cref{thm:coarseModuli}. 

Following the description of Geraschenko and Satriano \cite{GS, GSII}, toric stacks are global quotients $[X/G]$, where $X$ is a toric variety and $G$ is an algebraic torus. 
Note, for example, that all of the stacks arising in this paper have been of this form. 
To provide a combinatorial description of such a toric stack, we thus need to provide the fan of $X$ and the action of $G$ on $X$, leading to the following definition. 

\begin{defn}[\cite{GS}*{Definitions 2.4 and 2.5}] \label{def:stackDatum}
A \newterm{toric stack datum} $(\tN, N, \tSigma, \beta)$ is
\vspace*{-.7mm}
\begin{enumerate}[noitemsep]
\item two finitely generated free abelian groups $N$ and $\tN$,
\item a fan $\tSigma$ in $\tN_{\RR}\ce \tN \otimes_{\ZZ} \RR $,
\item a group homomorphism $\beta\colon \tN \to N$ with finite cokernel. 
\end{enumerate}
For the toric variety $X_{\tSigma}$ with maximal dense torus $T_{\tN} \ce \tN \otimes_{\ZZ} \CC^*$, $G_\beta\ce \ker \beta \otimes_{\ZZ} \CC^*$ is a subtorus of $T_{\tN}$. 
The toric stack associated to the toric stack datum $(\tN, N, \tSigma, \beta)$ is $[X_{\tSigma}/ G_\beta]$. 
\end{defn}

\begin{example}\label{ex:P112-stackyDatum}
Let $\tN =\ZZ^3$, $N=\ZZ^2$, $\beta = \left[\begin{smallmatrix}
1 & -1 & \phantom{-}0\\ 0 & \phantom{-}2 & -1
\end{smallmatrix}\right]$, $\tSigma$ be the face fan of $\RR_{\geq 0}^3$ minus the interior, and $\Sigma$ be the projection under $\beta_\RR$ of $\tSigma$. 
Then there is an exact sequence 
\[0\to G_\beta\xrightarrow{\iota} (\CC^*)^3 \to (\CC^*)^2\to 0,
\]
where $G_\beta\cong\CC^*$, and via this identification, $\iota(t)=(t,t,t^2)$. 
Then $\cX_{\tSigma, \beta} = [(\AA^3-\boldzero)/\CC^*]$ 
is the weighted projective stack $\PP(1,1,2)$. 
\end{example}

\begin{defn}[\cite{GS}*{Definition 3.2}]
\label{def:morph-stackData}
A \newterm{morphism of toric stack data} is a pair of maps $(\Phi, \phi)\colon (\tN, N, \tSigma, \beta) \to (\tN', N', \tSigma', \beta')$ is
a commutative diagram of abelian groups
\begin{center}
\begin{tikzcd}
\tN \ar[r, "\beta"] \ar[d, "\Phi"'] & N \ar[d, "\phi"] \\
\tN' \ar[r, "{\beta'}"] & N'
\end{tikzcd}
\end{center}
such that $\Phi\colon \tN \to \tN'$ induces a map of fans $\tSigma \to \tSigma'$.
\end{defn}

\begin{example}\label{ex:P112-morph}
We continue with \Cref{ex:P112-stackyDatum}, which involved the toric stack datum associated to the weighted projective stack $\PP(1,1,2)$.  Additionally, consider the toric stack datum associated to the weighted projective variety $\PP_v(1,1,2)$.  This consists of  $(\ZZ^2,\ZZ^2,\Sigma,\id_{\ZZ^2})$, where $\Sigma\subseteq \RR^2$ is the standard complete fan for this weighted projective space.  
A morphism between these toric stacks
$\PP(1, 1, 2)\to \PP_v(1, 1, 2)$ is given by the commutative diagram 
\begin{center}
\begin{tikzcd}
\ZZ^3 
\ar[r, "\beta"] 
\ar[d, "\beta"'] & \ZZ^2 \ar[d, "{\id}"] \\
\ZZ^2 \ar[r, "{\id}"] & \ZZ^2. 
\end{tikzcd}
\end{center}
\vspace*{-5mm}
\qedhere
\end{example}

\begin{defn}\label{def:toricStackDatum}
We now define a toric stack datum associated to our semiprojective toric variety $X$ with fan $\Sigma$. 
Continuing with the notation of \Cref{defn:intCoxAndCartierStack}, and let $\Sigma_{\max}$ denote the set of maximal faces of $\Sigma$. 
Consider $\ZZ^{\Sigma(1)}$ with basis $\{e_\rho\mid \rho\in\Sigma(1)\}$ and sublattice 
\begin{equation}\label{eqn:defnL}
\widetilde{L} \ce 
\ZZ^{\Sigma(1)}\text{-lattice saturation of }
\left\<
\sum_{\rho \in \sigma(1)} a_\rho e_\rho \ \in \ZZ^{\Sigma(1)}\Bigg\vert\, \sum a_\rho u_\rho = 0, a_\rho \in \ZZ, \sigma \in \Sigma
\right\>.
\end{equation}
The key point of this definition of $\widetilde{L}$ is that there is one generator for each dependence relation among the rays of each nonsimplicial cone $\sigma$ of $\Sigma$. 

Passing to the quotient by $\widetilde{L}$ introduces relations among the rays $e_\rho$ that mirror the nonsimplicial structure of $\Sigma$.  
Specifically, define $\tN\ce \ZZ^{\Sigma(1)} / \widetilde{L}$ and write $\overline{e_\rho}$ for the image of $e_\rho$ in $\tN$ and $\widetilde{\rho}$ for the ray spanned by $\overline{e_\rho}$. 
Consider the map
\begin{align*}
\beta\colon \tN = \ZZ^{\Sigma(1)} / \widetilde{L} &\to N
\\ 
\overline{e_\rho} &\mapsto u_\rho,
\end{align*}
where $u_\rho$ is the primitive lattice point on the ray $\rho \in N$. 
Finally, the fan $\tSigma$ on $\tN$ is obtained as follows: for each cone $\sigma \in \Sigma$, include the cone $\widetilde{\sigma}$  spanned by the rays $\widetilde{\rho}$ for $\rho\in \sigma(1)$.  

Modding out $\ZZ^{\Sigma(1)}$ by $\widetilde{L}$ ensures that the above definition of $\tSigma$ is in fact a fan; and by construction, its cones are in natural bijection with the cones of $\Sigma$.  
This provides \defi{the toric stack datum 
$(\tN, N, \tSigma, \beta)$ associated to $X$}. 
\end{defn}

\begin{example}
\label{ex:canonicalSimpicial}
If $X$ is a simplicial toric variety, then $\widetilde{L}=0$.  
Thus $\tN = \ZZ^{\Sigma(1)}$ and $\beta(e_\rho) =u_\rho$, so that the $\QQ$-Cartier stack $\sXCar=\sXcan$ recovers the canonical stack associated to $X$ (see~\cite{GS}*{Section~5}).  
In particular, in this case, $\sXcan$ is a smooth toric DM stack.
\end{example}

\begin{example}\label{ex:toricDataNonSimplicial}
As in \Cref{ex:motivatingEx}, let $X$ be the affine toric variety whose fan is a single cone spanned by rays $\rho_1 = (0,0,1), \rho_2 = (1,0,1), \rho_3 = (0,1,1)$, and $\rho_4 = (1,1,1)$ in $\RR^3$. 
This is the nonsimplicial toric variety $X\cong \Spec(\CC[a,b,c,d]/(ad-bc))$.  
In this case, $\widetilde{L} = \< e_1-e_2-e_3+e_4\>$, the quotient $\ZZ^{\Sigma(1)}/\widetilde{L}$ is naturally isomorphic to $N$, and $\tSigma$ is naturally identified with $\Sigma$ under this isomorphism.  
In particular, the $\QQ$-Cartier stack $\sXCar = X$, which is a singular variety (with no stackiness).
\end{example}

The remainder of this section is devoted to proving that the toric stack described by the data in \Cref{def:toricStackDatum}
 is equivalent to $\sX_{\QCart}$.  
 
\begin{prop} \label{prop:canonicalDMstack-toric}
The toric stack associated to the toric stack datum $(\tN, N, \tSigma, \beta)$ in \Cref{def:toricStackDatum} is isomorphic to the $\QQ$-Cartier stack $\sXCar$ of the toric variety $X$.   
\end{prop}

Let $\CDiv^\QQ_{T_N}(X) \subseteq \ZZ^{\Sigma(1)}$ denote the subgroup of $T$-equivariant Weil divisors on $X$ that are $\QQ$-Cartier. Observe that the $\QQ$-Cartier Cox ring $\checkS$ is the monoid ring given by
\[
\checkS = \CC[\CDiv^\QQ_{T_N}(X) \cap \NN^{\Sigma(1)}], 
\] 
which happens to be the ring for the affine toric variety associated to the dual cone 
\begin{align}\label{eq:sigmaCheck}
\sigma_{\checkS} \ce \RR_{\geq 0}(\CDiv^\QQ_{T_N}(X) \cap \NN^{\Sigma(1)})^\vee \subseteq \RR^{\Sigma(1)}. 
\end{align}
Let $\SF(\Sigma, N)$ denote the lattice of support functions on $\Sigma$ from \cite{CLSToricVarieties}*{Definition 4.2.11}. 

\begin{lemma} \label{lem:affineComparison}
The lattice $\widetilde N$ in \Cref{def:toricStackDatum} is dual to the lattice $\CDiv^\QQ_{T_N}(X)$.  As a consequence, the cone in \eqref{eq:sigmaCheck} satisfies $\sigma_{\checkS} = \RR_{\geq 0}(\overline{e_\rho}) = |\widetilde \Sigma|$ and $\wcheckG = \QCart(X)$.  
Furthermore, $\CDiv_{T_N}(X)$ is dual to 
\footnotesize
\[
\left\<
\sum a_\rho e_\rho \ \in \ZZ^{\Sigma(1)}\Bigg\vert\, \sum a_\rho\varphi(u_\rho) \in \ZZ, \forall \varphi \in \SF(\Sigma, N) 
\right\> \Bigg/ \left\<
\sum_{\rho \in \sigma(1)} a_\rho e_\rho \ \in \ZZ^{\Sigma(1)}\Bigg\vert\, \sum a_\rho u_\rho = 0, \sigma \in \Sigma
\right\>.
\]
\normalsize
\end{lemma}
\begin{proof}
First, note that $\CDiv_{T_N}(X)\cong\SF(\Sigma,N)$ by \cite{CLSToricVarieties}*{Theorem 4.2.12(c)}.  This identifies the saturation $\CDiv^{\QQ}_{T_N}(X)$ of $\CDiv_{T_N}(X)$ with the lattice of support functions $\SF^{\QQ}(\Sigma,N)$, which are only required to be integral on the rays of $\Sigma$. 
Observe that the map
\begin{align*}
\SF^{\QQ}(\Sigma, N) & \hookrightarrow \ZZ^{\Sigma(1)} \\
\varphi & \mapsto \sum \varphi(u_\rho)e_\rho \notag
\end{align*}
is injective, since piecewise linearity guarantees both that a support function is determined by its values on the $u_\rho$ and that $\sum\varphi(u_\rho)e_\rho$ pairs to zero with each vector of $\tL$. 

Now $\tL$ is a saturated sublattice of $\ZZ^{\Sigma(1)}$, so it is part of the split short exact sequence 
$
0\to\tL\to\ZZ^{\Sigma(1)}\to\tN\to 0. 
$ 
Dualizing, there is a containment 
\begin{align}\label{eq:SFQ}
\SF^\QQ(\Sigma,N) \subseteq 
\left\{\bba \in\ZZ^{\Sigma(1)}
\,\,\Big\vert\,\, \bba\cdot\bbv = 0\text{ for all }\bbv\in\tL\right\} = \tN^\vee.
\end{align}

Next, given $\bba =\sum a_\rho e_\rho \in\ZZ^{\Sigma(1)}$ with $\bba\cdot\bbv$ for all $\bbv\in\tL$, define an associated $\varphi_\bba\in\SF^\QQ(\Sigma,N)$ as follows. Given a rational vector $\bbp\in|\Sigma|$, let $\sigma_\bbp$ be the smallest cone in $\Sigma$ that contains $\bbp$ and write $\bbp = \sum_{\rho\in{\sigma_p}(1)}p_\rho u_\rho$. 
Then define $\varphi_\bba(\bbp) \ce \sum_{\rho\in\sigma_p(1)} a_\rho p_\rho$. 
Note that $\varphi_\bba$ is well-defined: 
if there are two faces $\sigma, \tau\in\Sigma$ with 
$\bbp = \sum_{\rho\in{\sigma}(1)}p_\rho u_\rho=\sum_{\rho\in{\tau}(1)}q_\rho u_\rho$, 
then, clearing the denominator, there is a $c\in\ZZ$ so that $c \cdot\sum (p_\rho-q_\rho) u_\rho\in\tL$. Now since $\bba\cdot\bbv$ for all $\bbv\in\tL$,  $\sum a_\rho p_\rho = \sum a_\rho q_\rho$, so $\varphi_\bba$ is well-defined.  
Therefore, it follows that the containment in \eqref{eq:SFQ} is an equality. 

We have now shown that there is a perfect pairing $\SF^{\QQ}(\Sigma, N) \otimes \widetilde N \to \ZZ$ that identifies the dual of the full sublattice $\SF(\Sigma, N) \subseteq \SF^{\QQ}(\Sigma, N)$ with
\footnotesize
\[
\left\<
\sum a_\rho e_\rho \ \in \ZZ^{\Sigma(1)}\Bigg\vert\, \sum a_\rho\varphi(u_\rho) \in \ZZ, \forall \varphi \in \SF(\Sigma, N) 
\right\> \Bigg/ \left\<
\sum_{\rho \in \sigma(1)} a_\rho e_\rho \ \in \ZZ^{\Sigma(1)}\Bigg\vert\, \sum a_\rho u_\rho = 0, \sigma \in \Sigma
\right\>,
\]
\normalsize
as desired.
\end{proof}

\begin{example}\label{ex:P112-SF}
Continuing \Cref{ex:P112-morph}, consider the weighted projective stack $\sXCar = \PP(1, 1, 2)$, where $\SF(\Sigma, N) = \{ (a,b,c) \in \ZZ^3 \ | \ a+b \equiv 0 \,\,(\mathrm{mod} \ 2) \}$,  while $\SF^{\QQ}(\Sigma, N) = \ZZ^3$. 
\end{example}

\begin{lemma} 
\label{lem:faceBijection}
For the cone $\sigma_{\checkS}$ from \eqref{eq:sigmaCheck}, there is a bijection 
\begin{align*}
\left\{\text{faces }\tau\text{ of } \sigma_{\checkS} \right\}
&\leftrightsquigarrow 
\left\{ 
\begin{matrix}\text{indexing sets } I \subseteq \Sigma(1) \text{ such that there exists a}\\ \QQ\text{-Cartier divisor }D = \sum_{\rho \in I} a_\rho D_\rho\text{ with }a_\rho \in \ZZ_{>0} 
\end{matrix}
\right\}
\\
\RR^{I^c}_{\geq 0} \cap \tau
&\ \mapsto \ \ 
I.
\end{align*}
\end{lemma}
\begin{proof}
By definition, the faces of $\tau^\vee$ are given by $\RR^I_{\geq 0} \cap \CDiv^\QQ_{T_N}(X)_{\RR}$.  Hence, these are manifestly indexed by such subsets $I$.  The written bijection comes from the bijection between faces of $\tau$ and complimentary faces of $\tau^\vee$.
\end{proof}

\begin{lemma} \label{lem:quasiaffineComparison}
The toric variety associated to $\widetilde \Sigma$ from \Cref{def:toricStackDatum} is $\Spec(\checkS) - V(\checkB)$.   
\end{lemma}
\begin{proof}
By \Cref{lem:affineComparison}, $\Spec(\checkS)$ is the affine toric variety $X_{\sigma_{\checkS}}$.
The fan associated to the quasi-affine toric variety $\Spec(\checkS) - V(\checkB)$ is a subfan of faces of $\sigma_{\checkS}$.  These faces are in bijection with indexing sets $I \subseteq \Sigma(1)$ such that there exists a $\QQ$-Cartier divisor $D = \sum_{\rho \in I} a_\rho D_\rho$ with $a_\rho \in \ZZ_{>0}$ by Lemma~\ref{lem:faceBijection}.  By definition of $\checkB$, the removed faces correspond to indexing sets $I$ which contain a primitive collection of $\Sigma$.  Since primitive collections exactly index the minimal sets of rays which are not cones, these are precisely the cones you do not have in  $\widetilde \Sigma$, as desired.
\end{proof}

\begin{proof}[Proof of \Cref{prop:canonicalDMstack-toric}]
This follows immediately from Lemmas~\ref{lem:affineComparison}~and~\ref{lem:quasiaffineComparison}.  
Namely, the $\QQ$-Cartier stack of $X$ by definition is $[\Spec(\checkS) - V(\checkB) / \checkG]$ which is the same thing as the toric stack associated to $(\widetilde N, N, \widetilde \Sigma, \beta)$ by these lemmas.  
\end{proof}

\section{The coarse moduli and universal properties for the \texorpdfstring{$\QQ$}{Q}-Cartier stack}
\label{sec:coarseModuli}

We now apply the construction of \Cref{sec:construction} to prove \Cref{thm:coarseModuli} and a subsequent stacky proof of \Cref{thm:nullstellensatz-var}. 
To compare $X$ and $\sXCar$, we rely on Kuwagaki's definition of a combinatorial equivalence \cite{Kuw20}.
\begin{defn} \label{def:combIso}
For arbitrary lattices $\tN, N$, a map $\beta\colon \tN \to N$ is a \defi{combinatorial equivalence of fans} $\tSigma$ and $\Sigma$ if the following two properties hold: 
\begin{enumerate}[noitemsep]
    \item The restricted map $\beta_{\RR}|_{\tsigma}\colon \tsigma \xrightarrow{\sim} \sigma$ is an isomorphism of cones for all $\tsigma\in \tSigma$.
    \item The map $\beta_{\RR}$ induces an isomorphism between the posets of cones given by
    \begin{align*}
    \vspace*{-3mm}
        \tSigma & \simeq \Sigma \\
        \tsigma & \mapsto \beta_{\RR}(\tsigma).
    \end{align*}
    \end{enumerate}
\end{defn}

This provides a combinatorial way to show that a toric stack is DM and that a map is a coarse moduli map.  It also clarifies what was ``wrong'' with using the canonical stack in the nonsimplicial case. 

\begin{example}\label{ex:nonCombEquiv}
Returning to Examples~\ref{ex:motivatingEx}~and~\ref{ex:toricDataNonSimplicial}, consider the nonsimplicial affine toric variety $X$ with fan a single cone spanned by $u_1 = (0,0,1), u_2 = (1,0,1), u_3 = (0,1,1)$, and $u_4 = (1,1,1)$ in $\RR^3$ and its faces. 
The associated toric stacky datum for the canonical stack $\sX_{\can}$ is $(\ZZ^4, \ZZ^2, \Sigma_{\can}, \beta_{\can})$, where $\Sigma_{\can}\subseteq \ZZ^4$ consists of a single four-dimensional cone spanned by $e_1, e_2, e_3$, and $e_4$ in $\RR^4$ and its faces. 
The map $\beta_{\can}\colon \RR^4 \to \RR^3$ is not a combinatorial equivalence; there is not a bijection of cones, since 
$\beta_{\mathrm{can}}$ maps the four-dimensional cone and all of the three-dimensional cones of $\Sigma_{\can}$ to the unique top-dimensional cone of $\Sigma$.
\end{example}

\begin{prop} 
\label{prop:combEquiv-implies-coarseModuli}
Let $\widetilde \Sigma$ be a fan in $\widetilde N_{\R}$ and $\Sigma$ be a fan in $N_{\R}$.
If $\beta\colon \widetilde N \to N$ is a combinatorial equivalence, then $[X_{\widetilde \Sigma}/G_\beta]$ is a DM stack and $\beta$ induces a coarse moduli map $[X_{\widetilde \Sigma}/G_\beta] \to X_\Sigma$. 
\end{prop}
\begin{proof}
The fact that $\sXCar$ is DM is shown in \cite{Kuw20}*{Lemma 5.4}.
  
To see that $\beta$ induces a coarse moduli morphism, first observe that any combinatorial equivalence is a good moduli morphism by \cite{GS}*{Theorem 6.3} (since a combinatorial equivalence trivially satisfies the conditions of the theorem). 
Now any good moduli morphism satisfies the universal property of a coarse moduli map by \cite{alper13}*{Theorem 4.16 vi}.

The other property of a coarse moduli map is that it induces a bijection on $\CC$-points.  
Since $\CC$ is algebraically closed, the only $G_\beta$-torsor over $\Spec(\CC)$ is $G_\beta$ itself.  By definition, the $\CC$-points of $[X_{\widetilde \Sigma}/G_\beta]$ correspond to $G_\beta$-equivariant maps $G_\beta \to X_{\widetilde \Sigma}$.  So $k$-points just correspond to $G_\beta$-orbits.

We now show that $G_\beta$-orbits are in bijection with points of $X$ since $\beta$ is a combinatorial equivalence.  Namely, since $\beta\colon \widetilde \Sigma \to \Sigma$ is a combinatorial equivalence,  the orbit-cone correspondence tells us that there is a bijection of torus orbits $O(\widetilde \sigma) \leftrightarrow O(\sigma)$.  Recall that for a cone $\sigma$, the stabilizer of the torus orbit $O(\sigma)$ is $N_\sigma \otimes_\ZZ \CC^*$ where $N_\sigma$ is the lattice spanned by the cone $\sigma$.  Now the map $\beta_\sigma\colon \widetilde N_{\tsigma} \to N_\sigma$  of lattices obtained by restricting $\beta$ becomes an isomorphism over $\R$ since $\beta$ is a combinatorial equivalence.  Therefore it is injective and the cokernel  $N_\sigma /\widetilde N_{\widetilde \sigma}$ is finite.  From the commutative diagram
\small
\[
\begin{tikzcd}
  & 0 \arrow[d]  \ar[r]
    & \widetilde N_{\tsigma} \arrow[d] \arrow[r, "{\beta_\sigma}"]
    & N_\sigma \arrow[d] \\
0 \arrow[r]
  & \ker \beta \arrow[d] \arrow[r]
  & \widetilde N \arrow[d] \arrow[r, "\beta"]
  & N \arrow[d] \arrow[r]
  & 0 \\
0 \arrow[r]
  & \ker \bar \beta \arrow[d] \arrow[r]
  & \widetilde N / \widetilde N_{\tsigma} \arrow[d] \arrow[r, "\bar \beta"]
  & N / N_\sigma \arrow[d] \arrow[r]
  & 0 \\
  & \ker \bar \beta /\ker \beta \arrow[r]
  & 0 \arrow[r]
  & 0,
\end{tikzcd}
\]
\normalsize 
so by the Snake Lemma, $\ker \bar \beta /\ker \beta \cong \widetilde N_{\tsigma} / \widetilde N_\sigma$ is a finite group. 
Therefore, tensoring the exact sequence 
$0 \to \ker \beta \to \ker \bar \beta \to \widetilde N/\widetilde N_{\tsigma} \to 0$ with $\CC^*$ yields a surjective map $\ker \beta \otimes_\ZZ \CC^* \to \ker \bar \beta \otimes_\ZZ \CC^* \to 0$.
Observe that the map $\bar \beta \otimes_{\ZZ} \CC^*$ can be identified with the map $O(\tsigma) \to O(\sigma)$ after choosing base points in both. 
Under this identification $\ker \bar \beta \otimes_\ZZ \CC^*$ is the fiber over a point in $O(\sigma)$ and $G_\beta = \ker \beta \otimes_\ZZ \CC^*$ is acting on the fibers via this surjective group homomorphism.  
In particular, the $G_\beta$-action preserves the fibers and acts transitively on them, establishing the desired bijection between $G_\beta$-orbits and points of $X$.
\end{proof}

Since the $\QQ$-Cartier stack $\sXCar$ of $X$ is encoded by the toric stack datum $(\widetilde{N}, N, \widetilde{\Sigma}, \id)$ by \Cref{prop:canonicalDMstack-toric}, the map $\sXCar \to X$ is given by the diagram
\begin{equation} 
\label{eq:coarseModuliMap}
\begin{tikzcd}
\tN \ar[r, "\beta"] \ar[d, "\beta"] & N \ar[d, "\id"] \\
N \ar[r, "{\id}"] & N.
\end{tikzcd}
\end{equation}

\begin{cor}\label{cor:DMandCoarseModuli}
   The $\QQ$-Cartier stack $\sXCar$ is Deligne-Mumford and the map $\sXCar\to X$  is a coarse moduli map.
\end{cor}
\begin{proof}
\Cref{prop:canonicalDMstack-toric} shows that $\sXCar$ is described by the toric stacky datum from \Cref{def:toricStackDatum} and the fact that $\beta\colon \tN \to N$ from \Cref{def:toricStackDatum} induces a combinatorial equivalence from $\tSigma$ to $\Sigma$ is immediate from the definition of $\tSigma$.  
It is immediate from the definition of $\tSigma$ that $\beta\colon \tN \to N$ from \Cref{def:toricStackDatum} induces a combinatorial equivalence from $\tSigma$ to $\Sigma$. 
Thus the desired result follows from \Cref{prop:combEquiv-implies-coarseModuli}.
\end{proof}

We now turn to \Cref{thm:coarseModuli}.(2), which is the universal property.

\begin{prop}\label{prop:univProperty}
Let $Y$ be an algebraic variety with a map $f\colon Y \to X$. If every $\QQ$-Cartier divisor on $X$ pulls back to a Cartier divisor on $Y$, then $f$ factors through $\sXCar$.
\end{prop}
\begin{proof}
Consider the relative spectrum of the sheaf of algebras 
\[
Y_{\checkG} \ce  \Spec\left(\bigoplus_{D \in \QCart(X)} \cO_Y(f^*D)\right), 
\]
and denote its global functions by 
the affinization 
$\check Y \ce \bigoplus_{D\in \QCart(X)} \Gamma(Y, \cO_Y(f^*D))$.

Now there is a homogeneous ring homomorphism
\begin{align*}
g^\sharp\colon \checkS &\to \check Y \\
s & \mapsto f^*s, 
\end{align*}
which induces a map $Y_{\checkG} \xrightarrow{a} \Spec(\check Y) \xrightarrow{g} \Spec(\checkS)$, where $a$ is the affinization map. 
Since $Y_{\checkG}$ is a $\checkG$-torsor over $Y$ by Lemma~\ref{lem:torsor}, there is a map
$h\colon Y \to [\Spec(\checkS) / \checkG]$.  
Thus by \Cref{thm:Pic-iso},  $X \cong [\Spec(S_{\Pic}) - V(B_{\Pic})/G_{\Pic}]$, where $G_{\Pic}\ce \Hom_\ZZ(\Pic(X),\CC^*)$.   
Under this isomorphism, the map $f$ corresponds to the diagram:
\[
\begin{tikzcd}
Y_{\Pic} \ar[r] \ar[d] & \Spec(S_{\Pic}) -V(B_{\Pic}) \ar[d] \\
Y \ar[r, dashed, "f"] & X, 
\end{tikzcd}
\]
where $Y_{\Pic}$ is the $G_{\Pic}$-torsor given by the relative spectrum
\[
Y_{\Pic} \ce \Spec\left(\bigoplus_{D \in \Pic(X)} \cO_Y(f^*D)\right). 
\]
This fits into a larger commutative diagram: 
\[
\begin{tikzcd}
Y_{\checkG} \ar[dr] \ar[r, "h"] \ar[ddr] & \Spec(\checkS) \ar[dr] \\
& Y_{\Pic} \ar[r] \ar[d] & \Spec(S_{\Pic}) -V(B_{\Pic})\ar[d] \\
& Y \ar[r, dashed, "f"] & X 
\end{tikzcd}
\]
It now follows that the image of $h$ lies in the preimage of $\Spec(S_{\Pic}) - V(B_{\Pic})$ in $\Spec(\checkS)$, i.e., it is contained in $\Spec(\checkS) - V(\checkB)$, completing the proof.
\end{proof}

\begin{lemma} \label{lem:torsor}
    The relative spectrum $Y_{\checkG}$ is a $\checkG$-torsor over $Y$.
\end{lemma}
\begin{proof}
Since $Y_{\checkG}$ is a relative spectrum, the fiber over $y \in Y$ is the affine variety 
\[
\Spec\left(\bigoplus_{D\in \QCart(X)} \cO_Y(f^*D)|_{y}\right).
\]
Since by assumption  $f$ satisfies the universal property, $\cO_Y(f^*D)$ is locally-free, so that this is isomorphic to 
$\Spec\left(\bigoplus_{D\in\QCart(X)} \CC \right) = \checkG$, as desired.
\end{proof}

Putting this all together, we are ready to prove \Cref{thm:coarseModuli}.

\begin{thmB}
If $X$ is a semiprojective toric variety, not necessarily simplicial, then the following hold. 
\begin{enumerate}
\item The $\QQ$-Cartier stack $\sXCar$ of $X$ is a toric\footnote{A purely toric construction of $\sXCar$ is given in \Cref{sec:construction}.} DM stack with coarse moduli space $X$. 
\item The stack $\sXCar$ is universal for the property that the pullback of any $\QQ$-Cartier divisor is Cartier. 
\item The stack $\sXCar$ is smooth if and only if $X$ is simplicial.
\end{enumerate}
\end{thmB}
\begin{proof}
Parts (1) and (2) follow from \Cref{cor:DMandCoarseModuli} and \Cref{prop:univProperty}, respectively.
For (3), if $X$ is simplicial, then \cite{CLSToricVarieties}*{Proposition 4.2.7} implies that $\sXCar \ce [\Spec (\checkS) - V(\checkB)/\checkG] = [\Spec(S)-V(B) / G]$, which is smooth.  
Conversely, if $X$ is not simplicial, then $\Spec(\checkS) - V(\checkB)$ has a nonsimplicial cone via the combinatorial equivalence. 
Hence, it is not smooth.
\end{proof}

\begin{remark}
We have now established a coarse moduli map $\pi\colon \sXCar \to X$. 
Equivalently, the notation $M|_X$ from \Cref{def:barNotation} can be defined to be $\pi_*(M|_{\sXCar})$.
\end{remark}

\begin{thmA}[Toric Ideal-Variety Correspondence]
If $X$ is a semiprojective toric variety, not necessarily simplicial, then there is a bijective correspondence between closed subvarieties of $X$ and radical homogeneous ideals $I \subseteq \checkB \subseteq \checkS$.
\end{thmA}
\begin{proof}[Stacky proof of Theorem~\ref{thm:nullstellensatz-var}]
This follows immediately from \Cref{thm:coarseModuli}, \Cref{cor:coarseModuli-bijection}, and \Cref{prop:canonicalDMstack-toric}. 
\end{proof}

\section{The Picard perspective}
\label{sec:picard}

Comparing with \Cref{thm:coarseModuli}, this section shows how a subring can be chosen that also provides a coarse moduli space for a semiprojective toric variety $X$ with fan $\Sigma$ and toric stack datum from \Cref{def:toricStackDatum}, 
which yields a generalization of \cite{cox:homog}*{Theorem 3.7}, see \Cref{cor:CoxQcoh}. 
Recall the definitions of $S_{\Pic}, B_{\Pic}$ and $G_{\Pic}$ from \Cref{defn:PicCox}, see also \Cref{ex:P113-3}.  
Note that the restriction of the Cox ring to Picard degrees was already considered in 
\cite{cox:homog} in the simplicial toric case and for more general settings in  
\cite{EKW-coxPic,AH2006-Tvar, brion07-wonderful, book2015-coxRings}. 

\begin{theorem} \label{thm:Pic-iso}
There is an isomorphism of stacks
\[
\psi\colon \left[\left(\Spec(S_{\Pic}) - V(B_{\Pic})\right) / G_{\Pic}\right] \to  X. 
\]
In particular, $X$ is a coarse moduli space for $\left[\left(\Spec(S_{\Pic}) - V(B_{\Pic})\right)/G_{\Pic}\right]$.
\end{theorem}
\begin{proof}
The affine toric varieties $\Spec(S_{\Pic}) - V(B_{\Pic})$ and $\Spec(\checkS) - V(\checkB)$ have the same fan but differ by the choice of underlying lattice. 
Namely, by \Cref{lem:affineComparison}, the lattice for $\Spec(S_{\Pic}) - V(B_{\Pic})$ is 
\scriptsize
\[
\widetilde N'\ce\left\<
\sum a_\rho e_\rho \ \in \QQ^{\Sigma(1)}\Big\vert\, \sum a_\rho\varphi(u_\rho) \in \ZZ, \forall \varphi \in \SF(\Sigma, N) 
\right\> \Bigg/ \left\<
\sum_{\rho \in \sigma(1)} a_\rho e_\rho \ \in \QQ^{\Sigma(1)}\Big\vert\, \sum a_\rho u_\rho = 0, \sigma \in \Sigma
\right\>. 
\]
\normalsize
Hence, by \Cref{prop:canonicalDMstack-toric}, the map $\psi$ comes from a combinatorial equivalence of fans.

Now, to show that $\psi$ is an isomorphism of stacks is equivalent to showing that $\Spec(S_{\Pic}) - V(B_{\Pic})$ is a $G_{\Pic}$-torsor over $X$.  
That means that locally on $U \subseteq X$,  $(\Spec(S_{\Pic}) - V(B_{\Pic}))|_U \cong U \times G_{\Pic}$.  
If $U = X_\sigma$ for some $\sigma \in \Sigma_{\max}$, then since $\psi$ is a combinatorial equivalence, its preimage $\psi^{-1}(X_\sigma)$ is given by the affine toric variety $X_{\widetilde \sigma}$ where $\widetilde \sigma = \RR_{\geq 0}\left(e_\rho \mid\rho \in \sigma(1)\right)$ in $\widetilde N'_{\RR}$. 

Now, we claim that the there is an isomorphism of lattices
\[
\beta\colon \widetilde N'_{\tilde \sigma}   \to N.
\] First, we claim the map is surjective.  Indeed, if $n \in \sigma$, then $n = \sum_{\rho \in \sigma(1)} a_\rho u_\rho$ for some $a_\rho \in \QQ_{\geq 0}$.  Then consider $n ' = \sum_{\rho \in \sigma(1)} a_\rho e_\rho$.  This lies in $\widetilde N'_{\widetilde \sigma}$ since support functions are linear on the cone $\sigma$ and have integer values on lattice points and it lies in the span of $\widetilde \sigma$ by definition.  Furthermore $\beta(n') = n$, so the map is surjective.

Now, by the claim, $\beta$ is a surjective map of lattices.  Furthermore, $\widetilde N'_{\widetilde \sigma}$ has the same rank as $N$ since $\beta$ is a combinatorial equivalence.  It follows that $\beta\colon \widetilde N'_{\tilde \sigma} \to N$ is an isomorphism (as any surjective map of lattices of the same rank is an isomorphism).

Finally, the split exact sequence
$
0 \to \Pic(X)^\vee \to \widetilde N' \to N \to 0
$
and the isomorphism of fans and lattices  $\beta$ up to the torus factor $G_{\Pic} = \Pic(X)^\vee \otimes \CC^*$ imply that $X_{\widetilde \sigma}$ is isomorphic to $X_{\sigma} \times G_{\Pic}$ as $G_{\Pic}$.  Hence $\Spec(S_{\Pic}) - V(B_{\Pic})$ is a $G_{\Pic}$-torsor over $X$, as desired.
\end{proof}

\begin{defn}\label{def:barNot-Pic}
For a finitely generated graded $S_{\Pic}$-module $M$, 
the associated coherent sheaf is $M|_X \ce \psi_* (M|_{\left[\left(\Spec(S_{\Pic}) - V(B_{\Pic})\right) / G_{\Pic}\right]})$.  
Equivalently, if $\sigma\in\Sigma_{\max}$, 
then for $\check{x}_{\sigma^c}$ in \Cref{lem:checkxsigma}, 
there is an equality 
$(S_{\Pic})_\sigma \ce S_{\Pic}[\check{x}_{\sigma^c}^{-1}]_0 = \cO_X(U_\sigma)$.
Thus for a finitely generated graded $S_{\Pic}$-module $M$, 
the associated coherent sheaf $M|_X$
is defined by $M|_X(U_\sigma) \ce M[\check{x}_{\sigma^c}^{-1}]_0$.
\end{defn}

\begin{cor} 
\label{cor:CoxQcoh} 
Let $X$ be a semiprojective toric variety, not necessarily simplicial. 
\begin{enumerate}[noitemsep]
\item 
For ideals $I, I'$ in $S_{\Pic}$, $I|_{X} = I'|_{X}$ if and only if $I\colon B_{\Pic}^\infty = I'\colon B_{\Pic}^\infty$.
In particular, closed subschemes of $X$ are in bijection with $G_{\Pic}$-invariant closed subschemes of $\Spec(S_{\Pic}) - V(B_{\Pic})$ and $B_{\Pic}$-saturated homogeneous ideals $I \subseteq S_{\Pic}$.
\item 
For a $\Pic(X)$-graded $S_{\Pic}$-module $M$,  $M|_{X}=0$ if and only if $B_{\Pic}^\ell M=0$ for some $\ell$.
\end{enumerate}
\end{cor}
\begin{proof}
This follows  from \Cref{thm:Pic-iso} and \Cref{cor:stackyNullstellensatz}.    
\end{proof}
\begin{remark}\label{rem:noKuwagaki}
Consider the diagram
\begin{equation} \label{eq:stackDiagram}
\begin{tikzcd} 
   {[\Spec(\checkS) - V(\checkB) / \checkG]} \ar[r, "f"] \ar[d, "\pi"]&  {[\Spec(S_{\Pic}) - V(B_{\Pic}) / G_{\Pic}]} \ar[dl, "\psi"]\\
   X
\end{tikzcd}
\end{equation}
where $f^\sharp\colon S_{\Pic}\hookrightarrow \checkS$ is the inclusion and $\psi$ is the isomorphism of stacks from \Cref{thm:Pic-iso}. This gives a more direct approach to proving that $[\Spec(\checkS) - V(\checkB) / \checkG]$ is a DM stack in \Cref{thm:coarseModuli} that does not rely on Kuwagaki's work. 
Observe in \eqref{eq:stackDiagram} that $f$ is a finite map (since it is induced by a finite map of algebras and a quotient map $\checkG \to G_{\Pic}$ with finite cokernel) and $\psi$ is an isomorphism by \Cref{thm:Pic-iso}.  Hence $[\Spec(\checkS) - V(\checkB) / \checkG]$ is finite over the algebraic variety $X$ and therefore is a DM stack.
\end{remark}

\begin{remark} \label{rem:veronese}
The map $f$ in \eqref{eq:stackDiagram} can be thought of as a multigraded analog of the Veronese embedding, 
see also \cite{book2015-coxRings}*{Construction 4.2.1.2 and Example 4.2.1.9}. 
\end{remark}

\section{A functorial stacky GKZ theory}
\label{sec:GKZ}

This section contains the primary application of our main results. 
GKZ theory studies a fan, denoted $\Sigma_{GKZ}$,  which parametrizes different GIT quotients of $\Spec(S)$.  Good references for this material include \cite{GKZbook} and \cite{CLSToricVarieties}*{Chapters~14-15}.  

GKZ theory functorially associates a toric GIT quotient to each cone $\Gamma \in \Sigma_{GKZ}$.  However, if one naively passes to the associated canonical stacks, even in the simplicial case, functoriality is lost. 

\begin{example}\label{ex:P113-H3-GKZcan}
The GKZ fan with the Hirzebruch surface $\cH_r$ as a chamber has one additional chamber corresponding to the smooth toric DM stack $\PP(1, 1, r)$, as pictured in Figure~\ref{fig:P113-GKZfan-var}.  
In this case, $\cH_r = \cH_{r,\mathrm{can}}$, but $\PP_v(1, 1, r)_\mathrm{can}$ is the stack $\PP(1, 1, r)$.  These two toric surfaces are related by a divisorial contraction, specifically one contracts the $-r$-curve on the Hirzebruch surface to a point.  On the other hand, this birational map does not lift to morphism of canonical stacks $\cH_3 = \cH_{3,\mathrm{can}} \to \PP(1, 1, 3) = \PP_v(1, 1, 3)_\mathrm{can}$.

\begin{figure}[htpb]
    \centering
    
    \begin{minipage}[c]{0.42\textwidth}
        \centering
        \begin{tikzpicture}[scale=0.8, baseline=(current bounding box.center)]
            \draw[step=1cm, gray!40, very thin] (0,2) grid (7,7);
            
            \fill[gray!15] (4,3) -- (0,13/3) -- (0,7) -- (4,7) -- cycle;
            \fill[gray!50] (4,3) -- (4,7) -- (7,7) -- (7,3) -- cycle;
            
            \draw[thick, black!60] (0,13/3) -- (0,7) -- (7,7) -- (7,3);
            
            \draw[->, blue!80!black, ultra thick, dotted] (4,3) -- (0,13/3); 
            \draw[->, blue!80!black, ultra thick] (4,3) -- (4,7);             
            \draw[->, blue!80!black, ultra thick, dashed] (4,3) -- (7,3);      
            
            \fill[black!60] (4,3) circle (4pt);
            
            \node at (2.2, 5.3) {${\mathsf{\PP}(1, 1, 3)}$};
            \node at (5.5, 5.3) {${\mathsf{\cH_3}}$};
        \end{tikzpicture}
    \end{minipage}%
    \hfill
    \begin{minipage}[c]{0.58\textwidth}
        \centering
        \begin{tikzcd}[row sep=3em, column sep={4em,between origins}]
                   & |[fill=gray!15, rounded corners, inner sep=4pt]| \PP(1, 1, 3) \arrow[dl] \arrow[dr] & & 
              |[fill=gray!50, rounded corners, inner sep=4pt]| \cH_3  
              \arrow[dl, red!80!black, line width=2.2pt, -{Latex[length=6pt, width=5pt]}, "\textbf{\textsf{X}}"{description, scale=1.6}]
              \arrow[dr] & \\
            |[draw=blue!80!black, ultra thick, dotted, rounded corners, inner sep=4pt]| \Spec(\CC) \arrow[drr] & & 
              |[draw=blue!80!black, ultra thick, rounded corners, inner sep=4pt]| \PP(1, 1, 3) \arrow[d] & & 
              |[draw=blue!80!black, ultra thick, dashed, rounded corners, inner sep=4pt]| \PP^1 \arrow[dll] \\
            & & |[fill=black!60, text=white, rounded corners, inner sep=4pt]| \Spec(\CC) & &
        \end{tikzcd}
    \end{minipage}
\caption{An illustration of the failure of functoriality for canonical stacks.}
\label{fig:P113-GKZfan-var}
\end{figure}
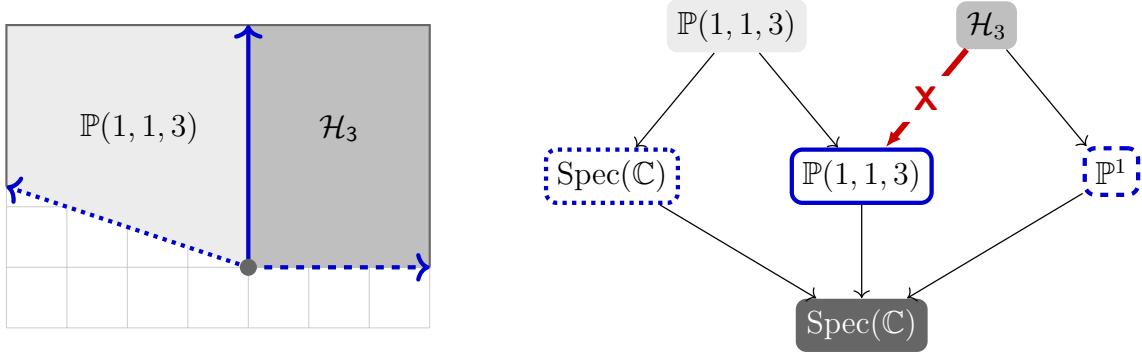

Indeed, this is impossible since under the map $\cH_3\to\PP_v(1,1,3)$, the line bundle $\cO_{\PP_v(1,1,3)}(r)$ pulls back to $\cO_{\cH_3}(H+rF)$, the hyperplane class $H$ plus $r$ times the fiber $F$. However, on the stack $\PP(1,1,3)$, $\cO_{\PP(1,1,3)}(1)$ is a line bundle. 
Hence, if the map is lifted to canonical stacks, then the pullback of $\cO_{\PP(1,1,3)}(1)$ would be $\frac{1}{r}H + F$, which is not an integral class, yielding a contradiction.  Alternatively, this can been seen directly in terms of the map not lifting to a morphism of stacky data.  This results in the failure of functoriality across $\Sigma_{GKZ}$ for the Hirzebruch surface $\cH_r$ pictured on the right hand side of Figure~\ref{fig:P113-GKZfan-var}. 
\end{example} 
  
Our entry point for GKZ theory starts with a diagonal action of an algebraic torus $G$ on an affine space $\AA^n$.  
This provides a short exact sequence 
$0\to G\to(\CC^*)^n\to(\CC^*)^n/G\to 0$.
Applying the exact functor $\Hom_\ZZ(-,\CC^*)$ yields 
\[
0\to M\xrightarrow{\beta^\vee} \ZZ^n\xrightarrow{\mu}\widehat{G}\ce \Hom_\ZZ(G,\CC^*)\to 0.
\]
Dualizing $\ZZ$ gives the exact sequence 
$\widehat{G}^\vee\to\ZZ^n\xrightarrow{\beta} N\to 0$,
and as motivating language, we call the set $\{\beta(e_1),\beta(e_2),\dots,\beta(e_n)\}$ the ``rays" of the GIT problem, even though these vectors might not come from a toric variety. 

To this data, there is an associated GKZ fan $\Sigma_{GKZ}$ that lies in vector space associate to the character lattice $\wG_\RR\ce\Hom_\ZZ(G,\CC^*)\otimes_\ZZ\RR$ of $G$, as defined by Gelfand--Kapranov--Zelevinsky \cite{GKZbook}, see 
\cite{CLSToricVarieties}*{Definition 14.4.2}. 
We now associate to each cone $\Gamma$ in $\Sigma_{GKZ}$ a toric stack $\sXGamma$, and for each containment $\Gamma' \subseteq \Gamma$, a morphism of toric stacks $\sX_{\Gamma'} \to \sXGamma$ which is functorial.  
When $\Gamma$ is in the moving cone, $\sXGamma$ will be the $\QQ$-Cartier stack of $X_\Gamma$, but in general, it will be something in between the two.  
Furthermore, this functoriality extends to the common stacky refinement of \cite{ballard2024king}.

Fix a cone $\Gamma = (\overline \Sigma_\Gamma, I_\Gamma) \in \Sigma_{GKZ}$.\footnote{Here we use $I_\Gamma$ instead of the $I_\emptyset$ used in \cite{CLSToricVarieties} to distinguish between these sets as they vary in $\Sigma_{GKZ}$.} 
Recall that $L_\Gamma$ denotes the lineality subspace of $\overline{\Sigma}_\Gamma$,
$N_\Gamma \ce N/(N \cap L_\Gamma)$ is a quotient lattice of $N$,
$\Sigma_\Gamma \ce \{\sigma/L_\Gamma \mid \sigma \in \overline{\Sigma}_\Gamma \}$ is the projection of $\overline{\Sigma}_\Gamma$ to $N_\Gamma$,
and $X_\Gamma \ce X_{\Sigma_\Gamma}$ is the toric variety associated to $\Sigma_\Gamma$.
Let $f_\Gamma \colon N \to N_\Gamma$
denote the natural projection.

Also of interest is the set of weights not appearing in either $\Sigma_{\Gamma}(1)$ or $I_\Gamma$, that is, the set 
\[
J_{\Gamma} \ce \{1, ..., n\} \setminus \left( \Sigma_{\Gamma}(1) \cup I_\Gamma(1)\right), 
\]
so that $I_\Gamma^c = \Sigma_\Gamma(1)\cup J_{\Gamma}$.  
In \eqref{eqn:defnL}, a sublattice $\widetilde{L}\subseteq \ZZ^{\Sigma(1)}$ was defined that encoded the nonsimplicial structure of $\Sigma$.  
We now define an analogous $\widetilde{L}_{\Gamma} \subseteq \ZZ^n$ that encodes not only the nonsimplicial structure of $\Sigma_{\Gamma}$, but also properties related to $I_\Gamma$: 
\small
\[
\tL_\Gamma \ce 
\text{lattice saturation of } \left\<
\sum_{
\rho\in \sigma_\Gamma(1) \cup J_\Gamma} a_\rho e_\rho \in \ZZ^n
\,\,\,\Bigg\vert\,\,\,
\sum a_\rho f_\Gamma(u_\rho) = 0, \, a_\rho \in \QQ, \, \sigma_\Gamma \in \Sigma_\Gamma
\right\>
\subseteq \ZZ^n. 
\]
\normalsize
Further, set
$\tN_\Gamma \ce \ZZ^n/\tL_\Gamma$ 
with natural quotient map
\begin{align*}
\tf_\Gamma\colon \ZZ^n & \to N \\
e_\rho & \mapsto u_\rho,
\end{align*}
 so that $\beta_\Gamma$ is the induced map on quotients in the following commutative diagram: 
\begin{equation}\label{eq:tildefGamma}
\begin{tikzcd}
\ZZ^n \ar[r, "{\beta}"] \ar[d, "{\tf_\Gamma}"] & N \ar[d, "{f_\Gamma}"] \\
\tN_\Gamma \ar[r, "{\beta_\Gamma}"] & N_\Gamma.
\end{tikzcd}
\end{equation}
Given a cone $\sigma_\Gamma \in \Sigma_\Gamma$, let $\tsigma_\Gamma$ denote the cone in $(\tN_\Gamma)_\R$ generated by the $\tf_\Gamma(e_\rho)$ for $\rho\in\sigma_\Gamma(1)$.

\begin{lemma}
The map $\beta_\Gamma$ induces an isomorphism of cones
$\tsigma_\Gamma \xra{\cong} \sigma_\Gamma$.
\end{lemma}
\begin{proof}
 The map $\beta_\Gamma$ has an inverse defined by $\sum a_\rho u_\rho \mapsto \sum a_\rho \tf_\Gamma(e_\rho)$.  The well-definedness of this map follows precisely from the definition of $\tN_\Gamma$.
 \end{proof}

The lemma allows us to define a fan
$
\tSigma_\Gamma \ce \{ \tsigma_\Gamma \mid \sigma_\Gamma \in \Sigma_\Gamma \}
$, 
as faces of such cones are now of this type as well.
\begin{defn}
\label{defn:goodStacky}
Let $\sXGamma$ be the toric stack defined via the toric stack datum
$
(\tN_\Gamma, N_\Gamma, \tSigma_\Gamma,  \beta_\Gamma).
$
\end{defn}

\begin{wrapfigure}[21]{r}{0.55\textwidth}
\raggedleft
\vspace*{-1mm}
\begin{center}
\hspace*{-2.5mm}
\begin{tikzpicture}
\draw[<->] (0,-2) -- (0,2);
\draw[->] (0,0) -- (-3,0);
\draw[->] (0,0) -- (2,2);
\draw[->] (0,0) -- (2,-2);
\node at (-3.5,0) {2};
\node at (0,-2.5) {4};
\node at (0,2.5) {1};
\node at (2.25,2.25) {3};
\node at (2.25,-2.25) {0};

\draw  (-2,-1) rectangle (-1,-2.5);
\node[black] at (-1.5,-2.75) {\tiny $I_\Gamma=\{2\}$};
\draw (-2,-1) -- (-1,-2.5);
\node at (-3,-2.5) {$\Sigma_{IV}$};
\draw  (-2,2) rectangle (-1,0.5);
\node[black] at (-1.5,0.25) {\tiny $I_\Gamma=\{2\}$};
\draw (-1,2) -- (-2,0.5);
\node at (-1.5,3) {$\Sigma_{III}$};

\draw  (0.5,3.5) rectangle (1.5,2);
\draw (0.5,2.75) -- (1.5,3.5) -- (0.5,2);
\node at (1.5,4) {$\Sigma_{II}$};

\draw  (2,1) rectangle (3,-0.5);
\draw (3,1) -- (2,0.25) -- (3,-0.5);
\node at (3.5,0.25) {$\Sigma_I$};

\draw  (1,-2.5) rectangle (2,-4);
\draw (1,-2.5) -- (2,-4) -- (1,-3.25);
\node at (2.5,-3.5) {$\Sigma_V$};
\draw[blue]  (-4.2,0.6) rectangle (-3.8,-0.2);
\node[blue] at (-4,-0.5) {\tiny $I_\Gamma=\{2\}$};
\draw[blue]  (-0.2,3.6) rectangle (0.2,2.8);
\draw[blue] (-0.2,2.8) -- (0.2,3.6);

\draw[blue] (2.6,3) rectangle (3,2.2);
\draw[blue] (2.6,2.6) -- (3,3);

\draw[blue]  (2.6,-1.6) rectangle (3,-2.4);
\draw[blue] (2.6,-2) -- (3,-2.4);

\draw[blue]  (-0.2,-2.8) rectangle (0.2,-3.6);
\draw[blue]  (-0.2,-2.8) -- (0.2,-3.6);
\end{tikzpicture}
\end{center}
\captionsetup{width=0.9\linewidth, justification=raggedleft}
\caption{For \cref{ex:5points}, $\Sigma_{GKZ}$ with $I_\Gamma=\varnothing$ unless stated otherwise.}
\label{fig:5pts-GKZcone}
\end{wrapfigure}

The next example illustrates that, in general, the stack $\sXGamma$ is something in between the Cox construction and the variety.  In this example, the Cox construction would not provide a combinatorial equivalence.

\begin{example}
\label{ex:5points}
In $M=\ZZ^3$, consider the rays which are the rows of the following matrix:
\[
\phi \ce \left[\begin{smallmatrix}1&0&0\\1&1&0\\1&0&1\\1&0&2\\1&1&2\end{smallmatrix}\right],
\hspace*{9.2cm}
\]
 so that
\[
0\to M\xrightarrow{\phi}\ZZ^5\xrightarrow{
    \left[\begin{smallmatrix} \phantom{-}1 & 0 & -2 & 1 & \phantom{-}0\\
    -1 & 1 & \phantom{-}0 & 1 & -1 \end{smallmatrix}\right]
}\ZZ^2\to 0
\hspace*{9.7cm}
\]
is the fundamental short exact sequence for a toric variety with such rays.
Let $\Gamma$ denote the cone in the GKZ fan that is a wall between the cones for $\Sigma_{I}$ and $\Sigma_{II}$, so that $\Sigma_\Gamma$ is as pictured in Figure~\ref{fig:5pts-GKZcone}.
In this case, $\beta\colon \ZZ^5\xrightarrow{\phi^T} N=\ZZ^3$,  $\tf_{I}=\tf_{II}=\id_{\ZZ^5}$, and $f_{I}=f_{II}=\id_{\ZZ^3}$.

\smallskip
Let $\tSigma$ denote the Cox fan lift of $\Sigma_\Gamma$ to $\RR^5$.
Here, $\tL_\Gamma$ is the span of one relation, given by $2e_0-e_1-2e_2+e_4$, which comes from the relation among the rays of the cone $\sigma$ of $\Sigma_\Gamma$ spanned by the four rays corresponding to rows 0, 1, 2, and 4 of $\phi$.
With the natural quotient by $\tL_\Gamma$ denoted by $\tf_\Gamma$, there is a commutative diagram
\[
\begin{tikzcd}[ampersand replacement=\&, column sep=3.5em]
\ZZ^5 \ar[rr, "{\beta=\phi^T = \left[\begin{smallmatrix} 1&1&1&1&1 \\0&1&0&0&1\\ 0&0&1&2&2 \end{smallmatrix}\right]}" {yshift = .5mm}] \ar[dd, swap, "{\tf_\Gamma = \left[\begin{smallmatrix} 1&0&0&0&-2 \\0&1&0&0&1\\ 0&0&1&0&2\\ 0&0&0&1&0 \end{smallmatrix}\right]}" {xshift=-1mm}] \& \& N \ar[dd, "{f_\Gamma}=\id_{\ZZ^3}" {yshift = 1mm}]
\\
\\
\tN_\Gamma \ar[rr,  "{ \beta_\Gamma = \left[\begin{smallmatrix} 1&1&1&1 \\0&1&0&0\\ 0&0&1&2 \end{smallmatrix}\right]}"] \& \& N_\Gamma.
\end{tikzcd}
\]
Note that $\beta_\Gamma\colon \tN_\Gamma\to N_\Gamma$ induces the combinatorial equivalence between $\tSigma_\Gamma$ and $\Sigma_\Gamma$.
\end{example}

As in \cite{CLSToricVarieties}*{(14.4.12)}, set 
$B(\Sigma_\Gamma,I_\Gamma) \ce \left\< \prod_{\rho\in\sigma(1)^c\cup I_\Gamma} x_\rho\ \Big\vert\ \sigma\in(\Sigma_\Gamma)_{\max} \right\>$, 
and define 
\[
S_\Gamma \ce \CC\left[(\mu_\RR^{-1}(\Span\Gamma))\cap\NN^{n}\right], 
\quad 
B_\Gamma \ce B(\Sigma_\Gamma,I_\Gamma)\cap S_\Gamma, 
\quad 
\text{and}
\quad 
G_\Gamma \ce \ker (\beta_\Gamma \otimes_{\ZZ} \CC^*).
\]

\begin{theorem} \label{thm:quotientRealization}
There is an isomorphism of toric stacks: 
\[
\sXGamma \cong [\Spec(S_\Gamma) - V(B_\Gamma)/G_\Gamma].
\]
\end{theorem}
\begin{proof}
The proof will largely echo the proof \Cref{prop:canonicalDMstack-toric}.
First, recall that by definition 
$\sX_\Gamma \ce [X_{\widetilde \Sigma_\Gamma}/G_\Gamma]$, so it suffices to show that the toric variety $X_{\widetilde \Sigma_\Gamma}$ is equal to $ \Spec(S_\Gamma) - V(B_\Gamma)$. 

Now $(\mu^{-1}(\Span\Gamma))\cap \ZZ^n$
is the lattice saturation of the GKZ cone: 
\begin{equation} 
\label{eq:gammaSupportFunctions}
\{(a_1, ..., a_r)\in \RR^n \mid \exists \varphi \in \SF(\Sigma_\Gamma) \text{ with } \varphi(u_\rho) = -a_\rho \ \forall \rho \notin I_\Gamma,  \varphi(u_\rho) \geq -a_\rho \ \forall \rho \in I_\Gamma \}. 
\end{equation}
(See, for example,  \cite{CLSToricVarieties}*{Definition 14.4.2}, and notice they are taking the interior of $\Gamma$, which is why they consider only convex support functions.)  Taking the argument from the proof of \Cref{lem:affineComparison} verbatim with $\widetilde L$ replaced by $\widetilde L_\Gamma$, it follows that $\Spec (S_\Gamma) = X_{|\widetilde \Sigma_\Gamma|}$, regarding $|\widetilde \Sigma_\Gamma|$ as a single cone. 
Given $\sigma\in(\Sigma_\Gamma)_{\max}$, the vanishing locus of 
$\left\<\prod_{\rho\in\sigma(1)^c\cup I_\Gamma} x_\rho\right\>\cap S_\Gamma$
is the closure of the orbit for $|\tSigma_\Gamma|$. 
Thus $\Spec (S_\Gamma) - V(B_\Gamma)$ is equal to $ X_{\widetilde \Sigma_\Gamma}$. 
\end{proof}

\begin{thmC}
For each cone $\Gamma$ in $\Sigma_{GKZ}$,  there is a toric DM stack $\sXGamma$ with a coarse moduli space $X_\Gamma$ such that, for any subcone $\Gamma' \subseteq \Gamma$, there is a functorial morphism that commutes with the usual functoriality of GKZ theory:
\[
\begin{tikzcd}
\cX_{\Gamma} \ar[r] \ar[d] & \cX_{\Gamma'} \ar[d] \\
X_{\Gamma} \ar[r] & X_{\Gamma'}.
\end{tikzcd}
\]
That is, there is a diagram $\cF$ from the poset category of cones in $\Sigma_{GKZ}$  to DM stacks and a natural transformation of diagrams 
\[
\begin{tikzcd}
\Sigma_{\text{GKZ}} \ar[r, "{\cF}"] \ar[dr, "F"'] & \{\text{DM stacks}\} \ar[d, "\text{coarse moduli}"] \\
& \{\text{Varieties}\},
\end{tikzcd}
\]
where $F$ is the usual diagram in varieties.
Further, $\sX_{\Gamma}$ lies in between $X_{\Gamma}$ and the $\QQ$-Cartier stack $\sXGC$ of $X_{\Gamma}$ via finite maps
$\sXGC \to \sXGamma \to X_{\Gamma}$,  
and when $\Gamma$ lies within the moving cone, there is an equality 
$\sX_{\Gamma}=\sXGC$.  
\end{thmC}
\begin{proof}
 By construction, $\beta_\Gamma\colon \widetilde N_\Gamma \to N_\Gamma$ induces a combinatorial equivalence of fans.  
 Hence $\sXGamma$ is DM, and $\sXGamma \to X_\Gamma$ is a coarse moduli map by \Cref{prop:combEquiv-implies-coarseModuli}.

To see that the assignment is functorial, 
given a subcone $\Gamma' \subseteq \Gamma$ in $\Sigma_{GKZ}$, there is a containment $L_{\Gamma'} \subseteq L_{\Gamma}$, and there is a quotient map $f_{\Gamma' < \Gamma} \colon N_{\Gamma}\to N_{\Gamma'}$ that lifts to a quotient map $\tf_{\Gamma' < \Gamma}\colon \tN_\Gamma \to \tN_\Gamma'$.
Together these fit into the following commutative diagram:
\begin{equation}
\label{eq:GKZmap}
\begin{tikzcd}
& & \tN_\Gamma \ar[r, "{\beta_{\Gamma}}"] \ar[dll, swap, "\beta_\Gamma"] \ar[drr, swap, "{\widetilde{f}_{\Gamma' < \Gamma}}"] & N_\Gamma \ar[dll, "{\id}"], \ar[drr, "{f_{\Gamma' < \Gamma}}"]& & \\
N_\Gamma \ar[r, swap, "{\id}" {yshift=1.5pt}] \ar[drr, swap, "{f_{\Gamma}}"] & N_\Gamma \ar[drr, "{f_{\Gamma}}"] & & & \tN_{\Gamma'} \ar[r, swap, "{\beta_{\Gamma'}}" {yshift=1.5pt}] \ar[dll, swap, "{\beta_{\Gamma'}}"] & N_\Gamma. \ar[dll, "{\id}"] \\
& & N_{\Gamma'} \ar[r, swap, "{\id}"] & N_{\Gamma'} & &
\end{tikzcd}
\end{equation}
Since $\overline{\Sigma}_{\Gamma}$ is a refinement of $\overline{\Sigma}_{\Gamma'}$, $\tSigma_{\Gamma}$ is a refinement of $\tSigma_{\Gamma'}$. Therefore, \eqref{eq:GKZmap} induces a functorial morphism of toric stacks:
\[
\begin{tikzcd}
\cX_{\Gamma} \ar[r] \ar[d] & \cX_{\Gamma'} \ar[d] \\
X_{\Gamma} \ar[r] & X_{\Gamma'}.
\end{tikzcd}
\]

We now show that there are finite maps
$\sXGC \to \sX_\Gamma \to X_\Gamma$, 
where $\sXGC$ is the $\QQ$-Cartier stack associated to $\Gamma$. 
For this, since the fan $\overline \Sigma_\Gamma$ is supported on the cone of the rays $\rho_1, ..., \rho_n$, each $\rho' \in J_\Gamma$ can be written as a $\QQ$-linear combination of rays in $\overline \Sigma(1)$, say 
\[
u_\rho = \textstyle\sum_{\rho' \in \overline \Sigma_\Gamma(1)} a_{\rho'\rho}u_{\rho'}.
\]
These give elements
$e_\rho - \sum_{\rho' \in \overline \Sigma_\Gamma(1)} a_{\rho'\rho}e_{\rho'} \ \in \ \widetilde L_\Gamma$. 
Now let $\widetilde L$ be the sublattice of $\ZZ^{\Sigma_\Gamma(1)}$ from the $\QQ$-Cartier stack construction. 
Consider the inclusion of lattices
$
i\colon \widetilde L \to \widetilde  L_\Gamma
$
given by restricting the inclusion $\ZZ^{\Sigma_\Gamma(1)} \to \ZZ^{I_\Gamma^c}$.  
Tensoring with $\QQ$ yields the injective linear map $i_{\QQ}\colon (\widetilde L)_{\QQ} \to (\widetilde  L_\Gamma)_{\QQ}$.  
Notice that a basis for $\widetilde L_{\QQ}$ can be extended to a basis for $(\widetilde L_\Gamma)_{\QQ}$ using the set
$\left\{ e_\rho - \sum_{\rho' \in \overline \Sigma_\Gamma(1)} a_{\rho'\rho}e_{\rho'} \ \Big\vert\ \rho \in J_\Gamma\right\}$.
This shows that the cokernel of $i$ has rank $|J_\Gamma|$. 

Now consider the short exact sequence of short exact sequences
\small
\[
\begin{tikzcd}
& 0 \arrow[d, shorten >=2pt]  \ar[r]
    & 0 \arrow[d, shorten >=2pt] \arrow[r]
    & 0 \arrow[d, shorten >=2pt] 
    \\[-2ex] 
0 \arrow[r]
  & \widetilde L \arrow[d] \arrow[r]
  & \ZZ^{\Sigma_\Gamma(1)} \arrow[d] \arrow[r, "\beta"]
  & \widetilde N \arrow[d, dashed, "\psi"] \arrow[r]
  & 0 \\
0 \arrow[r]
  & \widetilde L_\Gamma \arrow[d] \arrow[r]
  & \ZZ^{I_\Gamma^c} \arrow[d] \arrow[r, "\beta_\Gamma"]
  & N_\Gamma \arrow[d] \arrow[r]
  & 0 \\
 0 \ar[r] & \coker i \arrow[r] \ar[d, shorten >=2pt] 
  & \ZZ^{J_\Gamma} \arrow[r] \ar[d, shorten >=2pt] 
  & \coker \psi \ar[r] \ar[d, shorten >=2pt] & 0 \\[-2ex] 
  & 0 \ar[r] & 0 \ar[r] & 0. &
\end{tikzcd}
\]
\normalsize
Since the rank of $\coker i$ is $|J_\Gamma|$, $\coker \psi$ is a finite group.  Thus $\psi\colon \widetilde N \to N_\Gamma$ is the inclusion of a finite index subgroup.  

Observe that there are combinatorial equivalences $\widetilde \Sigma \cong \Sigma_\Gamma$ (induced by $\beta_\Gamma$) and $\widetilde \Sigma_\Gamma \cong \Sigma_\Gamma$ (by definition), which induce the maps
$\sXGC \xrightarrow{\psi} \sX_\Gamma \to X_\Gamma$.  
We showed that $\psi$ is a finite, surjective map above, and the composition $\sXGC\to X_\Gamma$ is finite by \eqref{eq:stackDiagram}.  
Thus $\sX_\Gamma \to X_\Gamma$ is also a finite map, 
see \cite{stacks-project}*{Tags 01W0 and 02LS}. 

Finally, if $\Gamma$ is in the moving cone, then by \cite{CLSToricVarieties}*{Proposition 15.1.3}, the span of $\Gamma$ is equal to $\QCart(X)$.  Therefore $S_\Gamma = \checkS$, and hence 
\begin{align*}
\sXGC & = \left[\left(\Spec (\checkS) - V(\checkB)\right)/\checkG\right] & \text{ by \Cref{prop:canonicalDMstack-toric}}&\\
& =  \left[\left(\Spec (S_\Gamma) - V(B_\Gamma)\right)/G_\Gamma\right] 
 = \sX_\Gamma & \text{by \Cref{thm:quotientRealization}.} &\qedhere
\end{align*}
\end{proof}

\begin{cor}  
The closed subvarieties of $X_\Gamma$ are in bijection with radical $\wG_\Gamma$-homogeneous ideals $I \subseteq B_\Gamma \subseteq S_\Gamma$, 
where $\wG_\Gamma\ce\Hom_\ZZ(G_\Gamma,\CC^*)$. 
\end{cor}
\begin{proof}
By \Cref{thm:mainGKZstacky} and \Cref{thm:quotientRealization}, $\Spec (S_\Gamma) - V(B_\Gamma) \to X_\Gamma$ is a coarse moduli map, so the result follows from \Cref{cor:coarseModuli-bijection}.
\end{proof}

\begin{cor}\label{cor:nullstellensatz-GKZ} 
Let $\Gamma$ be a cone in a GKZ fan. 
\begin{enumerate}[noitemsep] 
\item 
Every ideal sheaf on $\sXGamma$ has the form $I|_{\sXGamma}$ for some $I\subseteq S_\Gamma$, and $I|_{\sXGamma} = I'|_{\sXGamma}$ if and only if $I\colon B_\Gamma^\infty = I'\colon  B_\Gamma^\infty$.
Furthermore, $G_\Gamma$-invariant closed subschemes of $\Spec (S_\Gamma) - V(B_\Gamma)$ are in bijection with $B_\Gamma$-saturated homogeneous ideals $I \subseteq S_\Gamma$.
\item For every coherent sheaf $\cF$ on  $\sXGamma$, there is a finitely generated, multigraded $S_\Gamma$-module $M$, such that $\cF = M|_{\sXGamma}$.  Moreover, $M|_{\sXGamma}=0$ if and only if $B_\Gamma^\ell M=0$ for some $\ell$.
\end{enumerate}
\end{cor}

\begin{proof}
This follows from \Cref{thm:quotientRealization} and \Cref{cor:stackyNullstellensatz}.
\end{proof}

\begin{cor}\label{cor:SigmaGammaSimplicialCase}
If $\Gamma \in \Sigma_{\text{GKZ}}$ is a maximal cone, i.e., a chamber, then   
$\sX_{\Gamma,\mathrm{can}} = \sX_{\Gamma,\QCart} = \sX_\Gamma$. 
\end{cor}
\begin{proof}
Since $\Gamma$ is a chamber, $S_\Gamma = S$, and there is a usual toric stack associated to the GIT quotient.  
This is equivalent to the canonical toric stack obtained from the Cox construction, see, e.g., \cite{FK}*{Corollary 4.23}. 
Thus by \Cref{thm:quotientRealization}, $\sX_{\Gamma,\mathrm{can}} = \sX_\Gamma$.  
On the other hand, 
$\sX_{\Gamma,\mathrm{can}} = \sX_{\Gamma,\QCart}$ because  $\Sigma_\Gamma$ is simplicial.
\end{proof}

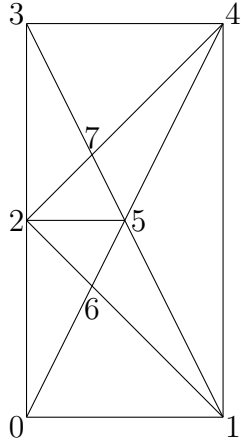
\begin{wrapfigure}[15]{l}{0.45\textwidth}
\raggedright
\vspace*{-4.5mm}
\begin{center}
\begin{tikzpicture}[scale=2.6]
\node at (-.05,-.05) {0};
\node at (1.05,-.05) {1};
\node at (-.05,1) {2};
\node at (-.05,2.05) {3};
\node at (1.05,2.05) {4};
\node at (.57,1) {5};
\node at (1/3,.55) {6};
\node at (1/3,1.42) {7};
\draw (0,0) -- (1,0);
\draw (0,0) -- (0,2);
\draw (0,0) -- (1,2);
\draw (0,2) -- (1,2);
\draw (0,2) -- (1,0);
\draw (0,1) -- (1,2);
\draw (0,1) -- (1,0);
\draw (1,0) -- (1,2);
\draw (0,1) -- (1/2,1);
\end{tikzpicture}
\end{center}
\captionsetup{width=0.95\linewidth, justification=raggedright}
\vspace*{-4.5mm}
\caption{The chosen common simplicial refinement in \Cref{ex:5points--wsX}.}
\label{fig:5pts-Lambda}
\end{wrapfigure}

In \cite{ballard2024king}, for the chambers $\Gamma\in\Sigma_{GKZ}$, a common toric stacky refinement $\widetilde{\sX}$ of all the $\cX_\Gamma$ was explicitly described as follows.  First, choose a common simplicial refinement $\Lambda$ of the collection of fans $\{\Sigma_\Gamma\}_\Gamma$. 
Now, since each $X_\Gamma$ is simplicial, any ray $\rho \in \Lambda$ lies in a unique minimal simplical cone $\sigma_{\rho, \Gamma} \in \Sigma_\Gamma$, and hence there is a relation
\[
\hspace*{8cm}
a_{\rho,\Gamma} u_\rho = \sum_{\tau \in \sigma_{\rho,\Gamma}(1)} a_{\tau} u_\tau,
\]
where $a_{\rho,\Gamma}, a_{\tau} \in \ZZ_{>0}$ and $\gcd(a_{\rho,\Gamma}, a_{\tau}) = 1$.  
Let $c_\rho \ce \operatorname{lcm}(a_{\rho,\Gamma})$, and define
\begin{align*}
\hspace*{8cm}
\beta_\Lambda\colon \ZZ^{\Sigma(1)} & \to N \\
e_\rho & \mapsto c_\rho u_\rho.
\end{align*}
This gives a toric stacky datum $(\ZZ^{\Sigma(1)}, N, \Lambda, \beta_{\Lambda})$.  Following \cite{ballard2024king}, define the following.
\begin{defn}
\label{def:tsX}
Let $\tsX$ be the toric stack from 
$(\ZZ^{\Sigma(1)}, N, \Lambda, \beta_{\Lambda})$,
where $\beta_{\Lambda}(e_\rho) \ce c_\rho u_\rho$.
\end{defn}

\begin{example}\label{ex:5points--wsX}
Returning to \cref{ex:5points}, there are two choices for $\Lambda$. Select the one shown in Figure~\ref{fig:5pts-Lambda}, which has the cone spanned by the rays with labels $2$ and $5$, as opposed to those labeled $6$ and $7$.
In this case,
\[
\begin{tikzcd}
\ZZ^{\Lambda(1)} \ar[r, "\beta_{\Lambda}"] \ar[d, "\alpha_i"] & N \ar[d, "id"]\\
\ZZ^5 \ar[r, "\beta_i"] & N,
\end{tikzcd}
\]
where $\beta_\Lambda = \left[\begin{smallmatrix}1&1&2&1&1&4&3&3\\
      0&1&0&0&1&2&1&1\\
      0&0&2&2&2&4&2&4      
      \end{smallmatrix}\right]$, 
$\beta=\phi^T = \left[\begin{smallmatrix} 1&1&1&1&1 \\0&1&0&0&1\\ 0&0&1&2&2 \end{smallmatrix}\right]$, 
and the following values for $\alpha_i$: 
\begin{center}
$\alpha_{I} = \left[\begin{smallmatrix}
      1&0&0&0&0&0&0&0\\
      0&1&0&0&0&1&1&0\\
      0&0&2&0&0&2&2&2\\
      0&0&0&1&0&0&0&0\\
      0&0&0&0&1&1&0&1
      \end{smallmatrix}\right]$,
\qquad 
$\alpha_{II} = \left[\begin{smallmatrix}
      1&0&0&0&0&2&2&0\\
      0&1&0&0&0&0&0&0\\
      0&0&2&0&0&0&0&2\\
      0&0&0&1&0&0&0&0\\
      0&0&0&0&1&2&1&1
      \end{smallmatrix}\right]$,
\qquad 
$\alpha_{III} = \left[\begin{smallmatrix}
      1&0&1&0&0&2&2&1\\
      0&1&0&0&0&0&0&0\\
      0&0&0&0&0&0&0&0\\
      0&0&1&1&0&0&0&1\\
      0&0&0&0&1&2&1&1
      \end{smallmatrix}\right]$, 
\end{center}
\begin{center}
$\alpha_{IV} = \left[\begin{smallmatrix}
      1&0&2&0&0&0&1&0\\
      0&1&0&0&0&2&1&1\\
      0&0&0&0&0&0&0&0\\
      0&0&2&1&0&2&1&2\\
      0&0&0&0&1&0&0&0
      \end{smallmatrix}\right]$, 
\qquad\text{and}\qquad
$\alpha_{V} = \left[\begin{smallmatrix}
      1&0&0&0&0&0&0&0\\
      0&1&0&0&0&2&1&1\\
      0&0&2&0&0&0&2&0\\
      0&0&0&1&0&2&0&2\\
      0&0&0&0&1&0&0&0
      \end{smallmatrix}\right]$.
\end{center}

\noindent
This data follows from the values of 
$a_{\rho,i}=1$ for all $\rho$ and $i$ except for $
a_{2,III}=a_{2,IV}=a_{5,I}=2$, 
so that $c_1 = c_3 = c_4 = 1$ and $c_2 = c_5 = 2$.
\end{example}

\begin{prop}\label{prop:DCox}
   Any choice of $\widetilde{\sX}$ as in \cite{ballard2024king} comes with maps $\alpha_\Gamma\colon \widetilde{\sX} \to \sXGamma$ such that the following diagram commutes:
   \[
   \begin{tikzcd}
       \widetilde{\sX} \ar[r, "{\alpha_\Gamma}"] \ar[dr, "{\alpha_{\Gamma'}}"'] & \sXGamma \ar[d] \\
       & \sX_{\Gamma'}. 
   \end{tikzcd}
   \]
   That is, the toric stack $\widetilde{\sX}$ is a cone for the GKZ diagram $\cF$ in toric stacks.
\end{prop}
\begin{proof}
This amounts to showing that if $\Gamma \subseteq \Gamma_i \cap \Gamma_j$ is a cone in $\Sigma_{GKZ}$ that is the intersection of two chambers, then there is a commutative diagram of toric stacks: 
\[
\begin{tikzcd}
& \tsX \ar[dl, swap, "{\alpha_i}"] \ar[dr, "{\alpha_j}"]& \\
\sX_{\Gamma_i} \ar[dr] & & \sX_{\Gamma_j}. \ar[dl]\\
& \sX_{\Gamma} &
\end{tikzcd}
\]
This amounts to commutativity of the following diagram of complexes of abelian groups: 
\[
\begin{tikzcd}
& & \ZZ^{\Lambda(1)} \ar[r, "{\beta_{\Lambda}}"] \ar[dll, swap, "\alpha_{\Gamma_i}"] \ar[drr, swap, "{\alpha_{\Gamma j}}"] & N \ar[dll, "{id}"], \ar[drr, "{id}"]& & \\
\ZZ^5 \ar[r, swap, "{\beta_{\Gamma_i}}"] \ar[drr, swap, "{\tf_{\Gamma}}"] & N \ar[drr, "{f_{\Gamma}}"] & & & \ZZ^5 \ar[r, swap, "{\beta_{\Gamma_j}}"] \ar[dll, swap, "{\tf_{\Gamma}}"] & N. \ar[dll, "{f_{\Gamma}}"] \\
& & \tN_{\Gamma} \ar[r, swap, "{\beta_{\Gamma}}"] & N_{\Gamma} & &
\end{tikzcd}
\]
The individual squares commute as morphisms of toric stacks, so it is left to check is that $\tf_{\Gamma} \circ \alpha_{\Gamma_i} = \tf_{\Gamma} \circ \alpha_{\Gamma_j}$.  
This holds because 
\begin{align*}
   \tf_{\Gamma} \circ \alpha_{\Gamma_i}(e_\rho) & = \frac{c_\rho}{a_{\rho,\Gamma_i}} \sum_{\tau \in \sigma_{\rho,\Gamma_i}(1)} a_\tau \tf_{\Gamma}(e_\tau) & \\
  & = \frac{c_\rho}{b_{\rho,\Gamma_j}} \sum_{\tau' \in \sigma_{\rho,\Gamma_j}(1)} b_{\tau'} \tf_{\Gamma}(e_{\tau'}) & \text{ since their difference is in the $\tL_\Gamma$,} \\
  & = \tf_{\Gamma} \circ \alpha_j(e_\rho).
  &&\qedhere
  \end{align*}
\end{proof}
\begin{example} \label{ex:hirzCones}
Returning to \Cref{ex:P113-H3-GKZcan}, as observed therein, the toric birational morphism $\cH_r \to \PP_v(1, 1, r)$ is given by contracting the $-r$-curve, which does not lift to a map of canonical stacks (which are the same thing as $\QQ$-Cartier stacks since these toric varieties are simplicial).   Instead, the birational correspondence comes from the root stack $\tsX$ of order $r$ over this $-r$-curve, which can be obtained as the toric stack coming from the map 
$\beta\colon \ZZ^4 \to N$ given by $\beta(e_i)=u_i$ for $i\ne 1$ and $\beta(e_1)=ru_1$. 
It also happens to be the fiber product of $\PP(1, 1, r)$ and $\cH_r$ over $\PP_v(1, 1, r)$.  Figure~\ref{fig:P113-again} illustrates how functoriality is restored and that the root stack $\widetilde \sX$ is a cone for the stacky Hirzebruch GKZ diagram.
\end{example}

\begin{figure}[htpb]
    \centering
    \begin{minipage}[c]{0.42\textwidth}   
        \centering
        \begin{tikzpicture}[scale=0.85, baseline=(current bounding box.center)]
            \draw[step=1cm, gray!40, very thin] (0,2) grid (7,7);
            
            \fill[gray!15] (4,3) -- (0,13/3) -- (0,7) -- (4,7) -- cycle;
            \fill[gray!50] (4,3) -- (4,7) -- (7,7) -- (7,3) -- cycle;
            
            \draw[thick, black!60] (0,13/3) -- (0,7) -- (7,7) -- (7,3);
            
            \draw[->, blue!80!black, ultra thick, dotted] (4,3) -- (0,13/3); 
            \draw[->, blue!80!black, ultra thick] (4,3) -- (4,7);             
            \draw[->, blue!80!black, ultra thick, dashed] (4,3) -- (7,3);      
            
            \fill[black!60] (4,3) circle (4pt);
            
            \node at (2.2, 5.3) {$\mathbf{\mathsf{\PP}(1, 1, 3)}$};
            \node at (5.5, 5.3) {$\mathbf{\mathsf{\cH_3}}$};
        \end{tikzpicture}
    \end{minipage}%
    \hfill
    \begin{minipage}[c]{0.58\textwidth}
        \centering
        \begin{tikzcd}[row sep=3em, column sep={4em,between origins}]
            & & \tsX \arrow[dl] \arrow[dr] & & \\
            & |[fill=gray!15, rounded corners, inner sep=4pt]| \PP(1, 1, 3) \arrow[dl] \arrow[dr] & & 
              |[fill=gray!50, rounded corners, inner sep=4pt]| \cH_3 \arrow[dl] \arrow[dr] & \\
            |[draw=blue!80!black, ultra thick, dotted, rounded corners, inner sep=4pt]| \Spec(\CC) \arrow[drr] & & 
              |[draw=blue!80!black, ultra thick, rounded corners, inner sep=4pt]| \PP_v(1, 1, 3) \arrow[d] & & 
              |[draw=blue!80!black, ultra thick, dashed, rounded corners, inner sep=4pt]| \PP^1 \arrow[dll] \\
            & & |[fill=black!60, text=white, rounded corners, inner sep=4pt]| \Spec(\CC) & &
        \end{tikzcd}
    \end{minipage}
    \caption{The $\sXGamma$ correction for the failure of functoriality in Figure~\ref{fig:P113-GKZfan-var}.}
    \label{fig:P113-again}
\end{figure}

\subsection*{Acknowledgments}
We are grateful to Isidora Bailly-Hall, Matthew Ballard, Michael K. Brown, Lauren Cranton Heller, Sheel Ganatra, Andrew Hanlon, Jesse Huang, Mykola Sapronov, Hal Schenck, and Gregory G.\ Smith for useful conversations. The authors benefited immensely from joint FRG support through NSF DMS 2412039--44. 
C. Berkesch was also supported by the Simons Foundation, 
D. Erman by NSF DMS 2246961 and 2200469, and D. Favero by NSF DMS 2302262.  

\bibliographystyle{amsalpha}
\bibliography{refs}

\end{document}